\documentclass[a4paper,12pt]{amsart}

\def\d{\partial}
\def\CP1{\mathbb{C}\mathrm{P}^1}
\def\C{\mathbb{C}}
\def\Z{\mathbb{Z}}
\def\Q{\mathbb{Q}}
\def\E{\mathcal{E}}
\def\tE{\tilde{\mathcal{E}}}
\def\oM{\overline{\mathcal{M}}}

\newcommand{\mZeta}[6]{\zeta\left( \begin{smallmatrix}% {\begin{array}{ccc}
 #1 & #2 & #3 \\
 #4 & #5 & #6
 \end{smallmatrix} % \end{array} } 
 \right)}

\newtheorem{theorem}{Theorem}[section]
\newtheorem{proposition}[theorem]{Proposition}
\newtheorem{corollary}[theorem]{Corollary}
\newtheorem{lemma}[theorem]{Lemma}

\theoremstyle{definition}
\newtheorem{definition}[theorem]{Definition}
\newtheorem{notation}[theorem]{Notation}

\theoremstyle{remark}
\newtheorem{remark}[theorem]{Remark}

\title{On double Hurwitz numbers with completed cycles}

\author{S.~Shadrin}

\author{L.~Spitz}

\address{S.~S. and L.~S.: Korteweg-de~Vries Institute for Mathematics, University of Amsterdam, P.~O.~Box 94248, 1090 GE Amsterdam, The Netherlands}

\email{s.shadrin@uva.nl, l.spitz@uva.nl}

\author{D.~Zvonkine}

\address{D.~Z.: Department of Mathematics, Stanford University, Stanford CA 94305, USA}

\email{dimitri.zvonkine@gmail.com}

\begin{document}

\begin{abstract}
In this paper, we collect a number of facts about double Hurwitz numbers, where the simple branch points are replaced by their more general analogues --- completed $(r+1)$-cycles. In particular, we give a geometric interpretation of these generalised Hurwitz numbers and derive a cut-and-join operator for completed $(r+1)$-cycles. We also prove a strong piecewise polynomiality property in the sense of Goulden-Jackson-Vakil. In addition, we propose a conjectural ELSV/GJV-type formula, that is, an expression in terms of some intrinsic combinatorial constants that might be related to the intersection theory of some analogues of the moduli space of curves. The structure of these conjectural ``intersection numbers'' is discussed in detail. \end{abstract}

\maketitle

\tableofcontents

\section{Introduction}

%In this paper, we collect a number of facts about double Hurwitz numbers, where the simple branch points are replaced by their more general analogues --- completed $(r+1)$-cycles. In particular, we prove a strong piecewise polynomiality property in the sense of Goulden-Jackson-Vakil. We also propose a conjectural ELSV-type formula, that is, an expression in terms of some intrinsic combinatorial constants that might be related to the intersection theory of some analogues of the moduli space of curves. The structure of these conjectural ``intersection numbers'' is discussed in detail. 

The theory of usual double Hurwitz numbers has been considered from several different points of view, usually in a purely combinatorial way. First, Okounkov~\cite{Oko00} observed a relation to integrable systems of KP-type that is also nicely explained and heavily used in the recent paper of Johnson~\cite{Joh10}. Second, double Hurwitz numbers were studied from different points of view in purely combinatorial way in a foundational paper of Goulden, Jackson, and Vakil~\cite{GouJacVak05}, where most of the typical contemporary questions were first posed. Third, there is a tropical approach to double Hurwitz numbers developed by Cavalieri, Johnson, and Markwig~\cite{CavJohMar10a,CavJohMar10b}. Fourth, double Hurwitz numbers were studied in genus $0$ through the intersection theory on double ramification cycles in the moduli space of curves by Shapiro, Vainshtein, and the first named author, see~\cite{Sha05,ShaShaVai08}.

The theory of Hurwitz numbers with completed cycles has emerged through several relations with geometry of the moduli space of curves. First, it has emerged in~\cite{OkoPan06} as a way to encode higher degrees of $\psi$-classes in the Gromov-Witten theory of $\CP1$. Second, it is related to the intersection theory of the moduli spaces of $r$-spin structures via an analogue of the ELSV formula~\cite{Zvo06}, conjectured by the third named author, that involves the classes discussed in~\cite{Chi08}. 

This motivates us to look attentively into the combinatorial structure of double Hurwitz numbers with completed cycles, since we expect that some kind of intrinsic combinatorial niceness of these numbers might be an explanation of the existing connections of completed cycles to geometry of moduli spaces and predict some further relations. So, we investigate double Hurwitz numbers with completed cycles in a purely combinatorial way generalizing various results in~\cite{GouJacVak05, Joh10, Sha09, ShaZvo07}. 

\subsection{Organization of the paper}

The paper starts with an introduction of the necessary theory. We define completed cycles as elements of the class algebra of the symmetric group using an isomorphism with the algebra of shifted symmetric functions in section~\ref{sectionCompletedHurwitz}. In section~\ref{sectionInfiniteWedge}, the infinite wedge space is explained and the necessary operators are introduced.

The main part of the paper consists of four sections. In section~\ref{sectionAlgorithm}, we introduce an algorithm to calculate Hurwitz numbers in a practical way, which allows us to prove the theorems in section~\ref{sectionStrongPiecewisePolynomiality}.

In section~\ref{sectionCutAndJoin}, we use the infinite wedge formalism together with calculations on the representation theory of the symmetric group to prove so called cut-and-join equations for completed Hurwitz numbers. These equations are direct generalizations of the the cut-and-join equation for ordinary Hurwitz numbers.

In section~\ref{sectionStrongPiecewisePolynomiality}, we use the algorithm defined in section~\ref{sectionAlgorithm} to prove a strong piecewise polynomiality property in the sense of \cite{GouJacVak05} for completed Hurwitz numbers. The proof is analogous to the proof in \cite{Joh10} for strong piecewise polynomiality of ordinary Hurwitz numbers. We also derive the corresponding wall crossing formula's.

Finally, in the last section we give a formula for completed one-part double Hurwitz numbers in terms of the intersection theory of some conjectural moduli spaces. We give an explicit formula for the generating series for these conjectural intersection numbers and we prove that it obeys the Hirota equations. The same was done for ordinary one-part double Hurwitz numbers in for instance ~\cite{Sha09}.

\subsection{Notation} By $[n]$, $n=1,2,\dots$, we denote the set $\{1,\dots,n\}$.

By $l(\mu)$ we denote the length of the partition $\mu=(\mu_1,\dots,\mu_l)$, and by $|\mu|$ we denote the sum $\sum_{i=1}^{l(\mu)} \mu_i$.

Let $U(z)$ be a formal power series in $z$. By $[z^a]U(z)$ we denote the coefficient of $z^a$ in $U$, that is, $U(z)=\sum_{a=1}^\infty z^a\cdot [z^a]U(z)$.

\subsection{Acknowledgements}
S.~S.~and L.~S.~were supported by a Vidi grant of the Netherlands Organization for Scientific Research.

%%%%%%%%%%%%%%%%%%%%
%%%%%%%%%%%%%%%%%%%%
%%%%%%%%%%%%%%%%%%%%

%%%%%%%%%%%%%%%%%%%%
%%%%%%%%%%%%%%%%%%%%
%%%%%%%%%%%%%%%%%%%%

%%%%%%%%%%%%%%%%%%%%
%%%%%%%%%%%%%%%%%%%%
%%%%%%%%%%%%%%%%%%%%

\section{Hurwitz numbers with completed cycles}\label{sectionCompletedHurwitz}

Following~\cite{OkoPan06} and~\cite{Zvo06}, we give a purely combinatorial definition of Hurwitz numbers with completed cycles. Intuitively, completed Hurwitz numbers count covers of $\CP1$ with given ramification over $0$ and~$\infty$, like ordinary Hurwitz numbers, but instead of allowing ramification given by ordinary branch points elsewhere, we allow ramification by completed cycles.

\subsection{Shifted symmetric functions}

Let $\Q[x_1, \ldots, x_n]$ be the algebra of polynomials in $n$ variables over~$\Q$. We define the shifted action of $S_n$ on this algebra by
\begin{equation}
\sigma (f(x_1 - 1, \ldots, x_n - n)) = f(x_{\sigma(1)} - \sigma(1), \ldots, x_{\sigma(n)} - \sigma(n)) 
\end{equation}
for any element $\sigma \in S_n$ and any polynomial $f$ written in terms of the variables~$x_i - i$. 
Denote by $\Q[x_1, \ldots, x_n]^{*S_d}$
the algebra of polynomials which are invariant with respect to this action. It has a natural filtration by degree. 

\begin{definition}
The algebra of shifted symmetric functions is the algebra
$\Lambda^* = \varprojlim \Q[x_1, \ldots, x_n]^{*S_d}$,
where the projective limit is taken in the category of filtered algebras with respect to the homomorphism which sends the last variable~$x_n$ to~$0$.
\end{definition}

In other words, the elements of the algebra $\Lambda^*$ are given as series 
$f = \{f^{(n)}\}$, $f^{(n)} \in \Q[x_1,\ldots, x_n]^{*S(n)}$, 
such that the polynomials $f^{(n)}$ are of uniformely bounded degree and stable under the restriction
$f^{(n+1)}|_{x_{n+1} = 0} = f^{(n)}$. 

\subsection{Basis of the algebra of shifted symmetric functions}

The shifted analogues of the power sums form a basis of the algebra of shifted symmetric functions. 

\begin{definition}\label{powerSum}
For any positive integer~$k$, the corresponding shifted symmetric power sum $p_k$ is defined as
\begin{equation}
p_k(\lambda) = \sum_{i=1}^{\infty}\left((\lambda_i - i + \frac{1}{2})^k - (-i + \frac{1}{2})^k\right) . 
\end{equation}
For any partition $\mu$, define $p_\mu = \prod_{i=1}^{l(\mu)} p_{\mu_i}$.
\end{definition}

The functions $p_\mu$ form a basis for the algebra of shifted symmetric funtions.  Another basis is obtained in the following way.

Irreducible representations of the symmetric group $S_d$ are in one to one correspondence with partitions of~$d$. 
On the other hand, partitions of~$d$ are also in one to one correspondence with conjugacy classes in~$S_d$. For partitions $\mu$ and $\lambda$, let $\dim(\lambda)$ be the dimension of the representation given by~$\lambda$, let $\chi^\lambda_\mu$ be the character of this representation evaluated on the conjugacy class $C_\mu$ given by~$\mu$, and let $|C_\mu|$ be the size of this conjugacy class. 

\begin{definition}
We define for any partition~$\mu$ a function~$f_\mu$ on the set of partitions of $|\mu|$ by
\begin{equation}\label{definitionFmu}
f_{\mu}(\lambda) = |C_\mu|\frac{\chi_{\mu}^{\lambda}}{\dim(\lambda)}. 
\end{equation}
When the partition $\mu$ is of the form $(r+1,1, \dots, 1)$, we denote $f_\mu$ by~$f_r$, $r=0,1,2,\dots$.
\end{definition}

Kerov and Olshanski proved that the functions $f_{\mu}$ are shifted symmetric~\cite{KerOls94}. They also form a basis for the space of shifted symmetric funtions. 

\subsection{Completed cycles}

Let $Z\C S_d$ be the class algebra of the symmetric group $S_d$. We can define a linear isomorphism
$\phi \colon \bigoplus_{d=0}^\infty Z\C S_d \to \Lambda^*$, $\phi\colon C_\mu \mapsto f_\mu$.

\begin{definition}
For any partition~$\mu$, the completed $\mu$-conjugacy class  $\overline{C}_\mu$ in the class algebra of the symmetric group is defined as
$\overline{C}_\mu := \phi^{-1}(p_\mu)/\prod_{i=1}^{l(\mu)} \mu_i!$.
A special role is played by the \emph{completed cycles}
$\overline{(r)} := \overline{C}_{(r)}$, $r = 1,2, \dots$.
\end{definition}

Note that this definition of the completed $(r+1)$-cycle differs by a factor $r!$ from the definition in~\cite{Zvo06}. 

The first few completed cycles are
\begin{align*}
0! \cdot \overline{(1)} &= (1)  \\ \notag
1! \cdot \overline{(2)} &= (2)\\ \notag
2! \cdot \overline{(3)} &= (3) + (1,1) + \frac{1}{12}\cdot(1) \\ \notag
3! \cdot \overline{(4)} &= (4) + 2\cdot(2,1) + \frac{5}{4}\cdot(2)  \\ \notag
4! \cdot \overline{(5)} &= (5) + 3\cdot(3,1) + 4\cdot (2,2) + \frac{11}{2}\cdot(3) \\
&+ 4\cdot (1,1,1) + \frac{3}{2} \cdot (1,1) + \frac{1}{80} \cdot (1) . 
\end{align*}

\subsection{Hurwitz numbers with completed cycles}

We fix two non-empty partitions $\mu$ and $\nu$ such that $|\mu|=|\nu|$ and non-negative integers $r$ and $s$ such that $2g-2+l(\mu)+l(\nu)=rs$.

\begin{definition} We define {\em disconnected double Hurwitz numbers with completed $(r+1)$-cycles} via the character formula:
\begin{equation} \label{eq:char-formula}
h^{(r)}_{g,\mu,\nu}:= \frac{1}{\prod_{i=1}^{l(\mu)} \mu_i \prod_{j=1}^{l(\nu)} \nu_j} \sum_{|\lambda| = d} \chi_\mu^\lambda \left(\frac{p_{r+1}(\lambda)}{(r+1)!}\right)^s  \chi_\nu^\lambda 
\end{equation}
\end{definition}

We also denote the same Hurwitz number by $h^{(r),s}_{\mu,\nu}$, and we often omit the superscript $(r)$ when $r$ is fixed in advance.

\begin{remark}
Since the completed $2$-cycle is equal to the ordinary $2$-cycle, the double Hurwitz numbers $h_{\mu,\nu}^{(1),s}$ are just the ordinary Hurwitz numbers for not necessary connected surfaces. 
\end{remark}

\subsection{Geometric interpretation}

Hurwitz numbers with completed cycles possess a geometric interpretation. 

Suppose we are given a partition $\lambda$ and a nonnegative interger $\gamma$. To this data we can assign a singularity of stable maps to $\CP1$ by the following rule. 
\begin{itemize}
\item
If $2 - 2\gamma - l(\lambda) <0$ then the singularity consists of a contracted curve of genus $\gamma$ (that is, a curve on which the stable map has degree 0) intersecting the remaining components of the source curve at $l = l(\lambda)$ branches on which the stable map has ramification points with indices $\lambda_1, \dots, \lambda_l$. Note that in this case the contracted curve is required to be connected, but not necessarily irreducible.

\item
If $\gamma=0$ and $l(\lambda) = 2$, then the singularity consists of a simple self-intersection of the source curve such that the stable maps presents ramification points with indices $\lambda_1$, $\lambda_2$.

\item
Finally, if $\gamma=0$ and $l(\lambda) = 1$, then the singularity is just a ramification point with index $\lambda_1$.

\end{itemize}

This list covers all possible connected singular loci of a stable map to $\CP1$.
In all three cases it is natural to consider the image of the singular locus under the stable map as a branch point of multiplicity $2\gamma + |\lambda| - l(\lambda)$, see~\cite{FanPan02}. For this reason we will say that the singular locus described by the data $(\lambda, \gamma)$ has multiplicity $2\gamma + |\lambda| - l(\lambda)$.

Now, a completed $(r+1)$-cycle is a linear combination of conjugacy classes $\lambda$. To each conjugacy class it is easy to assign a nonnegative integer $\gamma$ such that $2\gamma + |\lambda| - l(\lambda) = r$. Thus every term of the completed cycle corresponds to a singular locus of multiplicity~$r$. Moreover, the terms of the completed cycle cover all types of singular loci like that. The numerical coefficient of a partition $(\lambda)$ in the completed $(r+1)$-cycle is some kind of intersection  number on the moduli space $\oM_{\gamma,l(\lambda)}$ of contracted components; its nature is still not entirely clear.

We can now give the geometric meaning of Hurwitz numbers with completed cycles. Call a stable map $f:C\to \CP1$ an {\em $r$-covering} if (i)~it has a finite number of preimages of~$0$ and~$\infty$ and (ii)~all its singular loci have multiplicity $r$ except possibly the preimages of $0$ and $\infty$. Call the {\em weight of a singular locus} described by $(\lambda, \gamma)$ the coefficient of $(\lambda)$ in the completed $(r+1)$-cycle, where $r = 2\gamma + |\lambda| - l(\lambda)$. Call the {\em weight of an $r$-covering} the product of weights of its singular loci divided by the number of automorphisms of the covering. 

\begin{proposition}
The Hurwitz number $h^{(r)}_{g,\mu,\nu}$ is equal to the sum of weights of not necessarily connected $r$-coverings $f:C \to \CP1$ with $s$ fixed branch points, where the Euler characteristic of $C$ equals $2-2g$ and the number of branch points equals $s=(2g-2+l(\mu)+l(\nu))/r$.
\end{proposition}

This proposition leads to a natural definition of connected Hurwitz numbers: just replace ``not necessarily connected'' by ``connected'' in the above formulation. Connected Hurwitz numbers can be computed from the disconnected ones by the exclusion-inclusion formula, see also~\cite{OkoPan06,Zvo06}.

%%%%%%%%%%%%%%%%%%%%%%%
%%%%%%%%%%%%%%%%%%%%%%%
%%%%%%%%%%%%%%%%%%%%%%%

%%%%%%%%%%%%%%%%%%%%%%%
%%%%%%%%%%%%%%%%%%%%%%%
%%%%%%%%%%%%%%%%%%%%%%%

%%%%%%%%%%%%%%%%%%%%%%%
%%%%%%%%%%%%%%%%%%%%%%%
%%%%%%%%%%%%%%%%%%%%%%%

\section{Semi-infinite wedge formalism}\label{sectionInfiniteWedge}

In this section we sketch the theory of semi-infinite wedge space  following \cite{OkoPan06} and~\cite{Joh10}.

\subsection{Infinite wedge}

Let $V$ be an infinite dimensional vector space with basis labelled by the half integers. Denote the basis vector labelled by $m/2$ by $\underline{m/2}$, so $V = \bigoplus_{i \in \Z + \frac{1}{2}} \underline{i}$.

\begin{definition}
The semi-infinite wedge space is the span of all wedge products of the form
\begin{equation}\label{wedgeProduct}
\underline{i_1} \wedge \underline{i_2} \wedge \cdots
\end{equation}
for any decreasing sequence of half integers $(i_k)$ such that there is an integer $c$ (called the charge) with $i_k + k - \frac{1}{2} = c$ for $k$ sufficiently large.

 In this paper, we are mostly concerned with the zero charge subspace of the semi-infinite wedge space, which is the space of all wedge products of the form~\ref{wedgeProduct} such that 
\begin{equation}\label{zeroCharge}
i_k + k = \frac{1}{2}
\end{equation}
for $k$ sufficiently large. For brevity, we will call this space the infinite wedge space from now on. 
\end{definition}

\begin{remark}
An element of the infinite wedge space is of the form
$\underline{\lambda_1 - \frac{1}{2}} \wedge \underline{\lambda_2 - \frac{3}{2}} \wedge \cdots$
for some partition~$\lambda$. This follows immediately from condition~\eqref{zeroCharge}. Thus, we canonically have a basis for the infinite wedge space labelled by all partitions. The inner product associated with this basis will be denoted~$(\cdot,\cdot)$. 
\end{remark}

\begin{notation}
We denote by~$v_\lambda$ the vector labelled by a partition~$\lambda$. The vector labelled by the empty partition is called the vacuum vector and denoted by $|0\rangle = v_{\emptyset} = \underline{-\frac{1}{2}} \wedge \underline{-\frac{3}{2}} \wedge \cdots$. 
\end{notation}

\begin{notation}
If $\mathcal{P}$ is an operator on the infinite wedge space, then we define the vacuum expectation value of~$\mathcal{P}$ by
$\langle \mathcal{P}\rangle = \langle 0 |\mathcal{P}|0\rangle$,
where $\langle 0 |$ is the dual of the vacuum vector with resprect to the inner product~$(\cdot,\cdot)$, and called the covacuum vector. 
\end{notation}

\subsection{Operators} We now define some operators on the infinite wedge space.

\begin{definition} Let $k$ be any half integer. Then the operator $\psi_k$ is defined by
$\psi_k \colon (\underline{i_1} \wedge \underline{i_2} \wedge \cdots) \ \mapsto \ (\underline{k} \wedge \underline{i_1} \wedge \underline{i_2} \wedge \cdots)$. This operator acts on the whole semi-infinite wedge space (the sum of spaces with all charges). It increases the charge by $1$.

The operator $\psi_k^*$ is defined to be the adjoint of the operator $\psi_k$ with respect to the inner product~$(\cdot,\cdot)$.
\end{definition}

\begin{definition}
The normally ordered products of $\psi$-operators are defined in the following way
\begin{equation}
{:}\psi_i \psi_j^*{:}\ := \begin{cases}\psi_i \psi_j^*, & \text{ if } j > 0 \\
-\psi_j^* \psi_i & \text{ if } j < 0\ .\end{cases} 
\end{equation}
This operator does not change the charge and can be restricted to the infinite wedge space. Its action on the basis vectors $v_\lambda$ can be described as follows: ${:}\psi_i \psi_j^*{:}$ checks if $v_\lambda$ contains $\underline{j}$ as a wedge factor and if so replaces it by $\underline{i}$. Otherwise it yields~$0$. In the case $i=j > 0$, we have ${:}\psi_i \psi_j^*{:}(v_\lambda) = v_\lambda$ if $v_\lambda$ contains $\underline{j}$ and $0$ if it does not; in the case  $i=j < 0$, we have ${:}\psi_i \psi_j^*{:}(v_\lambda) = - v_\lambda$ if $v_\lambda$ does not contain $\underline{j}$ and $0$ if it does. These are the only two cases where the normal ordering is important.
\end{definition}

\begin{remark}
Let $E_{ij}$ for $i, j \in \Z + \frac{1}{2}$ denote the standard basis of matrix units of $\mathfrak{gl}(\infty) = \mathfrak{gl}(V)$. Then the assignment $E_{ij} \mapsto\ {:}\psi_i \psi_j^*{:}$ defines a projective representation of the Lie algebra $\mathfrak{gl}(V)$ on~$\Lambda^{\frac{\infty}{2}}(V)$. 
\end{remark}

\begin{notation}\label{notationZeta}
We denote by $\zeta(z)$ the function $e^{z/2} - e^{-z/2}$.
\end{notation}

\begin{definition}
Let $n \in \Z$ be any integer. We define two operators $\mathcal{E}_n(z)$ and $\tilde{\mathcal{E}}_n(z)$ depending on a formal variable~$z$ by
\begin{align*}
\mathcal{E}_n(z) &= \sum_{k \in \Z + \frac{1}{2}} e^{z(k - \frac{n}{2})} E_{k-n,k} + \frac{\delta_{n,0}}{\zeta(z)} \\
\tilde{\mathcal{E}}_n(z) &= \sum_{k \in \Z + \frac{1}{2}} e^{z(k - \frac{n}{2})} E_{k-n,k} \ .
\end{align*}
The operator $\mathcal{E}_n(z)$ is called the deregularized $\mathcal{E}$-operator, while $\tilde{\mathcal{E}}_n(z)$ is the regularized $\mathcal{E}$-operator.
\end{definition}

\subsection{Hurwitz numbers with completed cycles}

\begin{notation}\label{Not:alphaF} We denote by $\alpha_n$, $n\not= 0$, the operator $\E_n(0)$. We denote by $\mathcal{F}_{r+1}$, $r\geq 1$, the operator $[z^{r+1}]\tilde\E_0(z)$.
\end{notation}

\begin{proposition}\label{VacuumFormCompletedHurwitz}
A double Hurwitz number with completed cycles can be expressed as a vacuum expectation value in the infinite wedge space in the following way:
\begin{equation}\label{eq:oper-formula}
h_{\mu,\nu}^{(r),s} = \frac{\langle\prod_{i=1}^{l(\mu)}\alpha_{\mu_i} \mathcal{F}_{r+1}^s \prod_{i=1}^{l(\nu)} \alpha_{-\nu_i}\rangle}{\prod_{i=1}^{l(\mu)} \mu_i \prod_{j=1}^{l(\nu)} \nu_j}.
\end{equation}
\end{proposition}

\begin{proof} This is just a way to rewrite Equation~\eqref{eq:char-formula}. 

Indeed, the standard facts in the infinite wedge formalism are that 
$\prod_{i=1}^{l(\mu)} \alpha_{-\mu_i} | 0 \rangle = \sum_{|\lambda| = |\mu|} \chi_\mu^\lambda v_\lambda$, for any partition $\mu$, and
$\langle 0| \prod_{i=1}^{l(\mu)} \alpha_{\mu_i} v_\lambda = \chi_\mu^\lambda$, for any partitions $\lambda$ and $\mu$ such that $|\lambda|=|\mu|$. It is proved, e.~g., in~\cite{KacRai87} or~\cite{Joh10}.

Meanwhile, $v_\mu$ is an eigenvector of $\mathcal{F}_{r+1}$ with eigenvalue $p_{r+1}(\mu)/(r+1)!$, where $p_{r+1}$ is the shifted symmetric power sum defined in Definition~\ref{powerSum}, and $\mu$ is an arbitrary partition.

After these two observations, Equations~\eqref{eq:oper-formula} and~\eqref{eq:char-formula} are clearly equivalent.
\end{proof}

%%%%%%%%%%%%%%%%%%
%%%%%%%%%%%%%%%%%%
%%%%%%%%%%%%%%%%%%

%%%%%%%%%%%%%%%%%%
%%%%%%%%%%%%%%%%%%
%%%%%%%%%%%%%%%%%%

%%%%%%%%%%%%%%%%%%
%%%%%%%%%%%%%%%%%%
%%%%%%%%%%%%%%%%%%

\section{Computation of Hurwitz numbers}\label{sectionAlgorithm}

In this section we generalize the algorithm for computation of double Hurwitz numbers in~\cite{Joh10} to the case of double Hurwitz numbers with completed $(r+1)$-cycles. This algorithm presents several advantages with respect to direct copmputations with characters. In particular, it will allow us to prove the piecewise polynomiality of Hurwitz numbers. In Sections~\ref{Sec:OnePart} and~\ref{Sec:CoefCompCycles} we also show that in certain cases it leads to quite explicit expressions for Huwritz numbers, that cannot be directly deduced from the character formulas. 

Throughout the section we fix $r\geq 1$ (and therefore omit it in all notations). We also fix two partitions, $\mu$ and $\nu$, such that $|\mu|=|\nu|$, and an integer $s\geq 0$. They are the ramification profiles over two special points and the number of completed cycles of a particular Hurwitz number that we consider here. 

\subsection{Notations and properties of operators}
For any subsets $I\subset [l(\mu)]$, $J\subset [l(\nu)]$, and $K\subset [s]$ we introduce the following notation.
We denote by $\mu_I$, $\nu_J$, and $z_K$ the sums
$\mu_I :=\sum_{i\in I}\mu_i$, $\mu_J :=\sum_{j\in J}\nu_j$, and $z_K :=\sum_{k \in K} z_k$.
We denote by $\E(I,J,K)$ the operator
$\E_{\mu_I - \nu_J}\left(z_K\right)$. The same piece of notation we use also in the case of $\tilde\E$. 
Let $M\subset [l(\mu)]$, $N\subset [l(\nu)]$, and $L\subset [s]$. The last piece of notation we need is
\begin{equation}
\mZeta IJKMNL 
:= \zeta\left(\det \left( {\begin{smallmatrix}
 \mu_I - \nu_J & z_K \\
 \mu_M - \nu_N & z_L
 \end{smallmatrix} }\right)\right) ,
\end{equation}
where $\zeta(z) = e^{z/2} - e^{-z/2}$ is as defined in Notation~\ref{notationZeta}.

Using this notation, we can rewrite a Hurwitz number with completed cycles as
\begin{align}\label{algorithmVacuumExpectation}
h_{\mu,\nu}^{s} &= \frac{1}{\prod_{i=1}^{l(\mu)} \mu_i \prod_{j=1}^{l(\nu)} \nu_j}\cdot [z_1^{r+1} \cdots z_s^{r+1}] \\
\notag & \left\langle \prod_{i=1}^{l(\mu)}\mathcal{E}(\{i\},\emptyset,\emptyset)  \prod_{k=1}^{s}\tilde{\mathcal{E}}(\emptyset,\emptyset,\{k\})
\prod_{j=1}^{l(\nu)}\mathcal{E}(\emptyset,\{j\},\emptyset)\right\rangle .
\end{align}
By the symbol $[z_1^{r+1} \cdots z_s^{r+1}]$ we denote the coefficient of the monomial $z_1^{r+1} \cdots z_s^{r+1}$ in the formal power series which follows it. 

We will use a special case of the commutation relation~(2.17) in~\cite{OkoPan06} that in our notation is given by the following lemma.
\begin{lemma}\label{commutationRelation}
For any subsets $I, M \subset [l(\mu)]$, $J,N \subset [l(\nu)]$ and $K,L \subset [s]$ such that $I \cap M=J \cap N=K \cap L=\emptyset$, we have
\begin{align}\label{commutationEquation}
 \left[\E(I, J, K) ,\E(M, N, L) \right]
= \mZeta I J K M N L
\E\left(I \cup M, J \cup N, K \cup L\right).
\end{align}
\end{lemma}
\begin{remark}\label{regularizedCommutation}
Lemma~\ref{commutationRelation} is also true if one of the $\E$-operators on the left hand side of the Equation~\eqref{commutationEquation} is replaced by $\tilde\E$.
\end{remark}

\subsection{An algorithm for computation}\label{sec:alg}
We say that the operator $\mathcal{E}_n(z)$ has positive (resp., negative, zero) energy if the integer~$n$ is positive (resp., negative, zero).
We see immediately that $\mathcal{E}(I,J,K)|0\rangle$ (resp., $\langle 0| \mathcal{E}(I,J,K)$) is zero when $\mathcal{E}(I,J,K)$ has positive (resp., negative) energy. The vacuum expectation value in Equation~\eqref{algorithmVacuumExpectation} has operators of negative energy on the right and operators of positive energy on the left; by commuting them, we will be able to make use of that observation. 

\begin{remark}
Below we present an algorithm that computes the vacuum expectation value on the right hand side of Equation~\eqref{algorithmVacuumExpectation}. It will consist in commuting operators of negative energy to the left. Since $\tilde{\mathcal{E}}(\emptyset,\emptyset,k)|0\rangle = 0$, each of $\tilde\E$-operators must be involved in a commutator in our computations. Then Remark~\ref{regularizedCommutation} implies that we might as well have started with the vacuum expectation value
\begin{equation}\label{algorithmVacuumExpectation2}
\left\langle \prod_{i=1}^{l(\mu)}\mathcal{E}(\{i\},\emptyset,\emptyset)  \prod_{k=1}^{s}\mathcal{E}(\emptyset,\emptyset,\{k\})
\prod_{j=1}^{l(\nu)}\mathcal{E}(\emptyset,\{j\},\emptyset)\right\rangle .
\end{equation}
if we additionally demand that each of the zero energy operators will be involved in a commutator at a certain step of the algorithm. 
\end{remark}

Now we describe the algorithm. Note that the vacuum expectation value~\eqref{algorithmVacuumExpectation2} is of the form 
\begin{equation}\label{AlgorithmStep}
\prod_{q \in Q}\mZeta {F_q}  {G_q}  {H_q} {M_q} {N_q} {L_q}
\left\langle \prod_{t \in T} \mathcal{E}(I_t, J_t, K_t) \right\rangle
\end{equation}
for some finite index sets~$T$ and~$Q$ and subsets 
$I_t, F_q, M_q \subset [l(\mu)]$, 
$J_t, G_q, N_q \subset [l(\nu)]$, 
$K_t, H_q, L_q \subset [s]$.

At the beginning of any step in the algorithm we have vacuum expectation value of this form. Suppose that the operators in the vacuum expectation value of some step do not all have zero energy. Then the step will consist in the following actions. 

Let $t_0$ be the index of the leftmost operator of negative energy. If it is the leftmost operator, then this vacuum expectation value is zero by the remarks above and the algorithm terminates. If it is not, commute it to the left, that is, apply the equality 
\begin{align}
&  \E(I_{t_0 - 1}, J_{t_0 - 1}, K_{t_0 - 1})\E(I_{t_0},J_{t_0},K_{t_0}) \\ \notag
 & =\E(I_{t_0},J_{t_0},K_{t_0})\E(I_{t_0 - 1}, J_{t_0 - 1}, K_{t_0 - 1})\\ \notag
 &+ [\mathcal{E}(I_{t_0 - 1}, J_{t_0 - 1}, K_{t_0 - 1}),\mathcal{E}(I_{t_0},J_{t_0},K_{t_0})] .
\end{align}
The vacuum expectation resulting from the first (resp., second) term on the right hand side is called the passing (resp., commutator) term. By Lemma~\ref{commutationRelation}, the commutator in the commutator term is equal to 
\begin{equation}
\mZeta {I_{t_0 - 1}} {J_{t_0 - 1}} {K_{t_0 - 1}} {I_{t_0}}  {J_{t_0}}  {K_{t_0}}
\mathcal{E}(I_{t_0} \cup I_{t_0 - 1}, J_{t_0} \cup J_{t_0 - 1}, K_{t_0} \cup K_{t_0 - 1}).
\end{equation}
We now choose either the passing term or the commutator term and continue the algorithm with it. 

In the end we will sum over the contributions from all possible choices. Because both the passing term and the commutator term are again of the form~\eqref{AlgorithmStep}, we can iterate this procedure. The algorithm terminates when the result is zero because an operator of negative energy is on the far left, or one of the $\E(\emptyset,\emptyset,\{k\})$ wasn't commuted with any negative energy operators and moved to the far right, or an operator of positive energy is on the far right. It also terminates when all operators in the vacuum expectation have zero energy.

% The algorithm terminates when either the result is zero because an operator of negative energy is on the far left or one of $\E(\emptyset,\emptyset,\{k\})$ didn't commute with any negative energy operators and moved to the far right, or all operators in the vacuum expectation value have zero energy. 

Since taking the passing term results in having an operator of negative energy further to the left and taking the commutator term results in having one less operator in the vacuum expectation value, the algorithm will terminate for any values of the partitions $\mu$ and~$\nu$ and any non-negative integer~$s$.

\begin{remark} Since we demand $|\mu| = |\nu|$ (otherwise the Hurwitz number $h_{\mu,\nu}^s$ makes no sense), a vacuum expectation value with only one operator can only appear in the algorithm if this operator has zero energy.
\end{remark}

\subsection{Commutation pattern} Using the algorithm in the previous section, we can give an expression for a Hurwitz number in terms of $\zeta$-functions.

\begin{definition} A commutation pattern $P$ is a set of six-tuples of sets 
$\{(P_I^l, P_J^l, P_K^l, P_M^l, P_N^l, P_R^l)\}_{l \in L(P)}$
(where $L(P) := [|L|]$ is some index set) such that we get a non-vanishing contribution to the vacuum expection value~\eqref{algorithmVacuumExpectation} when we go through the algorithm in such a way that the $l$-th commutator computed is
$[\mathcal{E}(P_I^l, P_J^l, P_K^l)\ ,\ \mathcal{E}(P_M^l, P_N^l, P_R^l)].$
\end{definition} 

The set of all commutation patterns for given values of $\mu$, $\nu$, and $s$ is denoted by $CP_{\mu,\nu}^{s}$.

%\begin{remark}
%Note that a commutation pattern might at some point produce the operator $\zeta(0)\mathcal{E}_0(0)$ as the commutator $[\mathcal{E}_n(0),\mathcal{E}_{-n}(0)]$. This will not neccesarily terminate the algorithm, as we see in the description above. To get the contribution of such a commutation pattern to the final sum, we replace both zeros in the arguments by a formal variable~$t$ and take the limit as $t$ goes to~$0$.   
%\end{remark}

Note that the final vacuum expectation value in any commutation pattern~$P\in CP_{\mu,\nu}^{s}$ will always be the vacuum expectation of a product of zero energy operators, that is, it will always be of the form 
$\langle \prod_{t \in T(P)} \mathcal{E}_0(z_{S_t}) \rangle$
for some index set~$T(P)$ and some non-intersecting subsets ${S_t} \subset [s]$, $t\in T(P)$ whose union is equal to~$[s]$. 

%From now on we will consider this information to be part of the commutation pattern.

\begin{theorem}\label{disconnectedTheorem}
The Hurwitz number $h_{\mu,\nu}^s$ is given by the following formula:
\begin{align}\label{formulaHurwitz}
h_{\mu,\nu}^{s} & = \frac{1}{\prod_{i=1}^{l(\mu)} \mu_i \prod_{j=1}^{l(\nu)} \nu_j}[z_1^{r+1}\cdots z_s^{r+1}] \\
& \sum_{P \in CP_{\mu,\nu}^{s}}\prod_{t \in T(P)} \frac{1}{\zeta(z_{S_t})} \prod_{l\in L(P)}\mZeta 
{ P_I^l}{P_J^l}{P_K^l}{P_M^l}{P_N^l}{P_R^l}
 \notag
\end{align}
\end{theorem}

\begin{proof}
This follows immediately from Equation~\eqref{algorithmVacuumExpectation} and the description of the algorithm in Section~\ref{sec:alg}. We use that 
$\langle \prod_{t \in T(P)} \E_0(z_{S_t})\rangle = 1/\prod_{t \in T(P)} \zeta(z_{S_t})$ 
which follows directly from the definition of the operators~$\E_0(z)$. 
\end{proof}

\begin{remark}
In Theorem~\ref{disconnectedTheorem}, a factor $\zeta(0)\mathcal{E}_0(0)$ coming from the commutator $[\mathcal{E}_n(0),\mathcal{E}_{-n}(0)]$ in the contribution of a commutation pattern to the Hurwitz number should be interpreted in the following way. When the factor is produced, replace the zero in the argument of the $\zeta$-function by $n$ times a formal variable~$t$, and replace the zero in the argument of the $\mathcal{E}$-operator by~$t$. Then when the whole commutation pattern is completed, let $t$ go to zero. We see immediately that this is the same as replacing the commutator $[\mathcal{E}_n(0),\mathcal{E}_{-n}(0)]$ by the scalar operator~$n$ instead of by $\zeta(0)\mathcal{E}_0(0)$. This also agrees with the analysis in~\cite{OkoPan06} (Equation~2.19).
\end{remark}

\subsection{Connected Hurwitz numbers}
The algorithm also allows for the computation of the connected Hurwitz number. For this, we need one more definition. 

\begin{definition}
A commutation pattern~$P\in CP_{\mu,\nu}^{s}$ is called \emph{connected} if the set $T(P)$ consists of exactly one element. 
The set of all connected commutation patterns is denoted by~$CP_{\mu,\nu}^{s,\circ}$. 
\end{definition}

\begin{theorem}\label{connectedTheorem}
The connected Hurwitz number $h_{\mu,\nu}^{s,\circ}$ is equal to
\begin{align} \label{eq:connected}
h_{\mu,\nu}^{s,\circ} & = \frac{1}{\prod_{i=1}^{l(\mu)} \mu_i \prod_{j=1}^{l(\nu)} \nu_j}[z_1^{r+1}\cdots z_s^{r+1}] \\ \notag
& \frac{1}{\zeta(z_{[s]})} \sum_{P \in CP_{\mu,\nu}^{s,\circ}} \prod_{l\in L(P)}\mZeta
 {P_I^l} {P_J^l} {P_K^l} {P_M^l} {P_N^l} {P_R^l}.
\end{align}
\end{theorem}

\begin{proof} Since $\mathcal{E}_0(z)$ is a scalar operator, for any commutation pattern $P \in CP_{\mu,\nu}^{s}$ we have 
$\langle \prod_{t \in T(P)} \mathcal{E}_0(z_{S_t}) \rangle = \prod_{t \in T(P)} \langle  \mathcal{E}_0(z_{S_t}) \rangle$. 
Furthermore, when $|T(P)| \geq 2$, the operators from the start of the algorithm contributing to $\mathcal{E}_0(z_{S_t})$ for different~$t \in T(P)$ do not interact with each other at all. That is, a commutator term involving operators eventually contributing to~$\mathcal{E}_0(z_{S_t})$ for different~$t$ is never taken.   Therefore, for any integer $n \geq 1$, the contribution to $h_{\mu,\nu}^{s}$ by covers with at least $n$ connected components is given exactly by the contribution to Equation~\eqref{formulaHurwitz} of commutation patterns~$P$ for which $|T(P)| \geq n$.

That means that the exclusion-inclusion formula for the connected Hurwitz number in terms of ordinary Hurwitz numbers coincides precisely with the exclusion-inclusion formula for the contribution to Equation~\eqref{formulaHurwitz} by commutation patterns~$P$ with $|T(P)| = 1$  in terms of contributions by arbitrary commutation paterns. This completes the proof. 
\end{proof}

%In the next sections, we explore the consequences of Theorems~\ref{disconnectedTheorem} and~\ref{connectedTheorem} and the algorithm. 

\subsection{Example: one-part double Hurwitz numbers} \label{Sec:OnePart}
In the case of one-part double Hurwitz numbers, that is, $l(\nu)=1$, we have just one commutation pattern. Therefore, applying Equation~\eqref{formulaHurwitz} (or, equivalently, Equation~\eqref{eq:connected}), we obtain
\begin{equation}
h_{\mu,|\mu|}^s=
\frac{1}{|\mu| \prod_{i=1}^{l(\mu)} \mu_i}[z_1^{r+1} \cdots z_s^{r+1}] \frac{\prod_{k=1}^s \zeta(|\mu|z_k) \prod_{i=1}^{l(\mu)} \zeta(\mu_i z_{[s]})}{\zeta(z_{[s]})}
\end{equation}

\subsection{Example: coefficients of the completed cycles} \label{Sec:CoefCompCycles}

The coefficients of the completed cycles are obtained from Theorem~\ref{connectedTheorem} in the case $s=1$, $\nu = (1, \dots, 1)$. In that case there is only one connected commutation pattern, and we get
\begin{equation}
|C_\mu| h^{1,\circ}_{\mu, (1,\dots, 1)} =
\frac{1}{|\mu|!} [z^{r+1}] \zeta(z)^{|\mu|-1} \prod_{i=1}^{l(\mu)} \zeta(\mu_iz) ,
\end{equation}
(recall that $|C_\mu|$ is the size of the conjugacy class~$C_\mu$). This formula agrees with the one given in~\cite[Equation (0.22)]{OkoPan06} modulo the diffences in conventions and notation.

\section{The cut-and-join operators}\label{sectionCutAndJoin}

\subsection{Three vector spaces}
Let $p_1, p_2, \dots$ be an infinite sequence of formal variables. Given a permutation $\sigma \in S_n$ with cycle lengths $k_1, \dots, k_s$ denote by $p(\sigma)$ the product $p_{k_1} \cdots p_{k_s}$. If we assign the weight $k$ to the variable $p_k$, then the monomial $p(\sigma)$ is of total weight~$n$. The map $p$ is extended by linearity to an isomorphism 
\begin{equation}
p \colon Z \C S_n \mapsto \C_n[p_1, \dots, p_n],
\end{equation}
where $Z \C S_n$ is the center of the group algebra of the symmetric group $S_n$ and $\C_n[p_1, \dots, p_n]$ is the space of quasi-homogeneous polynomials of weight~$n$. This is an isomorphism of vector spaces, but not of algebras, since the target space of~$p$ does not have a natural algebra structure. However the action of elements of $Z \C S_n$ by multiplication gives rise to interesting operators in the space of homogeneous polynomials. In particular, the map $p$ transforms the operator of multiplication by the sum of all transpositions into the well-known cut-and-join operator:
\begin{equation}
\frac12\sum_{i,j \geq 1} \left(ij p_{i+j} \frac{\d^2}{\d p_i \d p_j} + 
(i+j) p_i p_j \frac{\d}{\d p_{i+j}} \right).
\end{equation}
The first term corresponds to the case where two cycles of lenghts $i$ and $j$ are merged together by the transposition; the second term corresponds to the case where a cycle of length $i+j$ is cut into two cycles of lengths $i$ and~$j$.

The operator corresponding to the mutiplication by the sum of all $(r+1)$-cycles was determined for every~$r$ by Goulden and Jackson in~\cite{GouJac}. Their expressions seem to be more complicated than those for the completed cycles.

The space $\C[[p_1, p_2, \dots]]$ of formal power series in variables $p_1, p_2, \dots$ is a completion of the direct sum of the spaces of quasi-homogeneous polynomials. It is natually isomorphic to a completion of the direct sum of spaces~$Z \C S_n$. Moreover, both vector spaces are naturally identified with the infinite wedge space via the isomorphism $v_\lambda \leftrightarrow s_\lambda(p)$, where $s_\lambda$ is the Schur polynomial:
\begin{equation}
s_\lambda(p) = \frac1{n!} \sum_{\sigma \in S_n} \chi_\lambda(\sigma) p(\sigma).
\end{equation}
Under this identification, the multiplication by the completed $(r+1)$-cycle corresponds to the operator 
\begin{equation}
F_{r+1} = \frac1{(r+1)!} \sum_{m \in \Z + 1/2} m^{r+1} E_{m,m}
\end{equation}
in the infinite wegde space. Our goal is to construct the corresponding cut-and-join operator in the space $\C[[p_1, p_2 \dots, ]]$.

\subsection{The construction of operators}
For $k \geq 1$, let $a_{-k} = p_k$ be the operator of multiplication by $p_k$ and $a_k = k \, \partial / \partial p_k$. We let $a_0=0$. In the infinite wedge space the operator $a_k$ becomes $\alpha_k=\tilde\E_k(0) = \sum_{m \in \Z +1/2} E_{m-k,m}$ (see Notation~\ref{Not:alphaF}). For $k<0$ it transforms $v_\lambda$ into $\sum \epsilon(\mu) v_\mu$, where the Young diagrams $\mu$ are all diagrams that can be obtained from $\lambda$ by adding a ribbon of length~$k$ and the sign $\epsilon(\mu)$ is the number of horizontal steps in the ribbon, as in the Murnaghan-Nakayama rule~\cite{Jam78}. Similarly, for $k>0$ it transforms $v_\lambda$ into $\sum \epsilon(\mu) v_\mu$, where the Young removing a ribbon of length~$k$ and the sign $\epsilon(\mu)$ is the number of horizontal steps in the ribbon.

The \emph{normal ordering} ${:}a_{k_1} \cdots a_{k_n}{:}$ of a monomial $a_{k_1} \cdots a_{k_n}$ is the non-decreasing order of indices; in other words the derivations go to the right and the multiplication operators to the left. Recall (Notation~\ref{notationZeta}) that $\zeta(z) = e^{z/2} - e^{-z/2}$.

\begin{definition}
The coefficients $Q_1, Q_2, \dots$ of the series
\begin{equation}
Q_1 z +Q_2 z^2 + \dots = \frac1{\zeta(z)} \sum_{n \geq 1} \frac1{n!} \sum_{k_1 +\dots + k_n =0}
\zeta(k_1 z) \cdots \zeta(k_n z)
\; \frac{{:}a_{k_1} \cdots a_{k_n}{:}}{k_1 \cdots k_n} 
\end{equation}
are called the \emph{completed cut-and-join operators}.
\end{definition}

\begin{theorem}\label{Thm:cutandjoin1}
The map $p: Z \C S_n \to \C_n[p_1, \dots, p_n]$ transforms the operator of multiplication by the completed $(r+1)$-cycle into the $(r+1)$st completed cut-and-join operator.
\end{theorem}

For instance, we have
\begin{align*}
Q_1 &= \sum_{i \geq 1} i p_i \frac{\d}{\d p_i},\\
Q_2 &= \frac12\sum_{i,j \geq 1} \left(ij p_{i+j} \frac{\d^2}{\d p_i \d p_j} + 
(i+j) p_i p_j \frac{\d}{\d p_{i+j}} \right),\\
Q_3 &= \frac16 \sum_{i,j,k \geq 1} \left(ijk p_{i+j+k} 
\frac{\d^3}{\d p_i \d p_j \d p_k} + 
(i+j+k) p_i p_j p_k \frac{\d}{\d p_{i+j+k}} \right) \\ 
&+ \frac14 \sum_{i+j=k+l} ij p_k p_l \frac{\d^2}{\d p_i \d p_j}
+ \frac1{24} \sum_{i\geq 1} (2i^3-i) p_i \frac{\d}{\d p_i}.
\end{align*}

The multiplication of $\sigma \in S_n$ by the completed cycle $\overline{(1)} = (1)$ corresponds to picking an element of~$\sigma$, which just mutiplies the permutation by~$n$. Hence the operator $Q_1$ multiplies a homogeneous polynomial by its total weight. 

The operator $Q_2$ is the standard cut-and-join operator.

The operator $Q_3$ is the more complicated cut-and-join operator whose action corresponds to multiplying a permutation by the completed 3-cycle. Let us briefly explain how its terms are related to the expression of the completed 3-cycle $\frac12 (3) +\frac12 (1,1) + \frac1{24} (1)$.

The last term $(1)$ is the operator of picking a sheet of the ramified covering or an element of the permutation. So $\frac1{24} (1)$ corresponds to 
\begin{equation}
\frac1{24} \sum i p_i \frac{\d}{\d p_i}.
\end{equation}
Here $i$ is the length of the cycle of the permutation that contains the picked element.

The second term $(1,1)$ is the operator of picking two sheets of the covering or two elements of the permutation. So $\frac12 (1,1)$ corresponds to
\begin{equation}
\frac14 \sum_{i,j \geq 1} ij p_i p_j \frac{\d^2}{\d p_i \d p_j}
+ \frac14 \sum_{i \geq 1} i(i-1) p_i \frac{\d}{\d p_i}.
\end{equation}
The first sum describes the case when the chosen elements belong to two different cycles of lengths $i$ and~$j$, while the second sum describes the case where they lie in the same cycle of length~$i$.

Finally, the term $(3)$ is the operator of multiplication by a 3-cycle. It's action is more complicated. If the elements of the 3-cycle belong to three different cycles of the permutations of lenghts $i,j,k$, these cycles are merged into one. This gives us the term
\begin{equation}
\frac16 \sum_{i,j,k \geq 1} ijk p_{i+j+k} \frac{\d^3}{\d p_i \d p_j \d p_k}.
\end{equation}
If one element lies in one cycle and two other elements lie in another cycle, then a piece of the second cycle is cut off and attached to the first one. But there is a subtlety: there are two ways to go from cycles of lengths, say, 2 and 19 to cycles of lengths 9 and 12: one can either take a 7-elements piece of the 19-cycle and attach it to the 2-cycle, or one can take a 10-elements piece. On the other hand, to go from cycles of lengths 2 and 19 to cycles of lengths 2 and 19 again there is only one way: one should take a 17-element piece from the 19-cycle and attach it to the 2-cycle. As a result, we get the following sum:
\begin{equation}
\frac14 \sum_{
\substack{i+j=k+l\\
\{i,j\} \not= \{k,l\}}
} ij p_k p_l \frac{\d^2}{\d p_i \d p_j}
+
\frac14 ij p_i p_j \frac{\d^2}{\d p_i \d p_j}.
\end{equation}
Finally, all three elements of the 3-cycle can lie in the same cycle of the permutation. If the cycles ``turn'' in two opposite ways, the cycle of the permutation is split into three parts and we get the term
\begin{equation}
\frac16 \sum_{i,j,k \geq 1} (i+j+k) p_i p_j p_k \frac{\d}{\d p_{i+j+k}} .
\end{equation}
If both cycles ``turn'' in the same direction, then the cycle of the permutation remains in one piece, though the order of elements changes. This corresponds to the operator
\begin{equation}
\frac1{12} \sum_{i \geq 1} i(i-1)(i-2) p_i \frac{\d}{\d p_i}.
\end{equation}
The reader can check that if we add all these terms we recover the operator~$Q_3$.

\subsection{The generating series for Hurwitz numbers}
Introduce the following generating series for the disconnected Hurwitz numbers with completed cycles:
\begin{align} 
& H_{r+1}(\beta, p_1, p_2, \dots, q_1, q_2, \dots)  = \\ \notag &
\sum_{n,m,s} \; \sum_{
\substack{\mu_1, \dots, \mu_m\\ \nu_1, \dots, \nu_n}
}
h^{(r+1)}_{g,\mu,\nu} \; \frac{\beta^s}{s!} \;
\frac{p_{\mu_1} \cdots p_{\mu_m}}{m!} \;
\frac{q_{\nu_1} \cdots q_{\nu_n}}{n!}.
\end{align}
Here, as before, $s = (2g-2+m+n)/r$ and, by convention, the summands with $\sum \mu_i \not= \sum \nu_i$ are set to~$0$.

\begin{theorem}\label{Thm:cutandjoin2}
The series $H_{r+1}$ satisfies the partial differential equation
\begin{equation}
\frac{\partial H_{r+1}}{\partial \beta} = Q_{r+1} H_{r+1}.
\end{equation}
\end{theorem}

This theorem is actually an equivalent formulation of Theorem~\ref{Thm:cutandjoin1}.

\subsection{Proofs}
Now we are going to prove Theorem~\ref{Thm:cutandjoin2} and hence the equivalent Theorem~\ref{Thm:cutandjoin1}.

According to Proposition~\ref{VacuumFormCompletedHurwitz}, we have
\begin{equation} \label{Eq:H}
H_{r+1} = \left< \exp\left(\sum_{k \geq 1} p_k \alpha_k/k\right) \exp(\beta F_{r+1}) \exp\left(\sum_{k \geq 1} q_k \alpha_{-k}/k\right) \right>.
\end{equation}
Hence,
\begin{equation} \label{Eq:dhdbeta}
\frac{\d H_{r+1}}{\d \beta} = \left<\exp\left(\sum_{k \geq 1} p_k \alpha_k/k\right) F_{r+1} \exp(\beta F_{r+1}) \exp\left(\sum_{k \geq 1} q_k \alpha_{-k}k/k\right) \right>.
\end{equation}

We will prove several lemmas that will allow us to simplify the above expression and to relate it to the cut-and-join operator $Q_{r+1}$. First of all recall (Lemma~\ref{commutationRelation}) that 
\begin{equation}
[\tE_a(z), \tE_b(w)] = \zeta(aw-bz)\tE_{a+b}(z+w)
\end{equation}
and (Notation~\ref{Not:alphaF}) that
\begin{equation}
\alpha_k = \tE_k(0),  \qquad  \tE_0(z) = \sum F_{r+1} z^{r+1}.
\end{equation}

\begin{lemma} \label{Lem1}
We have
\begin{align}
& \exp\left(\sum_{k \geq 1} p_k \alpha_k/k\right) \tE_0(z) \exp\left(-\sum_{k \geq 1} p_k \alpha_k/k\right) 
\\ \notag &
= \tE_0(z) + \frac1{1!} \sum_{i=1}^\infty \zeta(iz) \frac{p_i}{i} \tE_i(z)  +\frac{1}{2!}\sum_{i,j=1}^\infty \zeta(iz)\zeta(jz)\frac{p_i}{i}\frac{p_j}{j} \tE_{i+j}(z)
\\ \notag &
+ \frac{1}{3!}\sum_{i,j,k=1}^\infty \zeta(iz)\zeta(jz)\zeta(kz)\frac{p_i}{i}\frac{p_j}{j}\frac{p_k}{k}\tE_{i+j+k}(z) + \dots
\end{align}
\end{lemma}

\begin{proof}
The map $u \mapsto e^x u e^{-x}$ is the exponent of the map $u \mapsto [x,u]$. Substituing $u = \tE_0(z)$, $x = \sum_{k \geq 1} p_k \alpha_k/k$ and using the commutation relations for the operators $\tE$ we obtain the above formula.
\end{proof}

\begin{lemma} \label{Lem2}
Let
\begin{equation}
O_K(z)=\sum_{n=1}^\infty \frac1{n!} \sum_{k_1+\cdots+k_n=K} \frac{\zeta(k_1z)\cdots\zeta(k_nz)}{\zeta(z)} \frac{\d}{\d p_{k_1}}\cdots \frac{\d}{\d p_{k_n}}.
\end{equation}
Then we have
\begin{equation}
\sum_{r \geq 0} Q_{r+1} z^{r+1} 
= \sum_{n=1}^\infty \frac{1}{n!} \sum_{k_1,\dots,k_n>0} \zeta(k_1z)\cdots\zeta(k_nz)\frac{p_{k_1}}{k_1}\cdots \frac{p_{k_n}}{k_n} O_{k_1+\dots+k_n}(z).
\end{equation}
\end{lemma}

\begin{proof}
This is obtained by a simple computation.
\end{proof}

\begin{lemma} \label{Lem3}
Let $X$ be any operator in the infinite wedge space independent of $p_1$, $p_2$, \dots. Then
\begin{equation}
\left<\tE_K(z) \exp\left(\sum_{k \geq 1} p_k \alpha_k/k\right) X \right>
= O_K(z) \left<\exp\left(\sum_{k \geq 1} p_k \alpha_k/k\right) X \right>.
\end{equation}
\end{lemma}

\begin{proof}
It is enough to prove the lemma under the assumption that $X(v_\emptyset) = v_\lambda$. The general case is obtained by taking a linear combination of operators $X$ like that. Let $|\lambda| = N$. 

We are going to evaluate the right-hand side of the equality and simplify it finally obtaining the left-hand side. Note that the vacuum expectation value $\langle \exp\left(\sum_{k \geq 1} p_k \alpha_k/k\right) v_\lambda \rangle$ is equal to the Schur polynomial $s_\lambda$ which is conveniently written as 
\begin{equation}
\frac1{N!} p \left(\sum_{\sigma \in S_N} \chi_\lambda(\sigma) \sigma \right).
\end{equation}
The action of the operator $O_K(z)$ has a natural interpretation in terms of permutations: the operator picks (in all possible ways) a set of cycles of $\sigma$ with total length~$K$. These cycles will be called distinguished. A distinguished cycle of length $k$ is assigned a factor of $\zeta(kz)$. A non-distinguished cycle of length~$k$ is assigned a factor $p_k$ as before. To $\sigma$ is assigned the product of these factors. And the result of the action of $O_K(z)$ is the sum of the contributions of all permutations $\sigma \in S_N$ divided by $N!$ and by $\zeta(z)$.

The summation over all permutations $\sigma \in S_N$ with a set of distinguished cycles of total length~$K$ can be replaced by a summation over the permutations whose distinguished cycles cover the elements from $1$ to~$K$. A permutation like that actually lies in $S_K \times S_{N-K}$. The contribution of permutations like that is $N!/K!(N-K)!$ times smaller than the contribution of all permutations, so the new sum should be divided by $K! (N-K)!$ instead of $N!$.

We can decompose the representation $\lambda$ of $S_N$ into a direct sum of representations of $S_K \times S_{N-K}$ as follows:
\begin{equation}
\bigoplus_{\mu \subset \lambda} \mu \otimes (\lambda \setminus \mu).
\end{equation}
Here $\mu$ denotes an irreducible representation of $S_K$ corresponding to a Young diagram included in $\lambda$ and $\lambda \setminus \mu$ is the (possibly reducible) representation $\mbox{Hom}_{S_K} (\mu, \lambda)$ of $S_{N-K}$. Using this decomposition we obtain:
\begin{equation}
{\rm RHS} = \frac1{\zeta(z)} \sum_{\mu \subset \lambda}
s_\mu(\zeta(z), \zeta(2z), \dots) \cdot s_{\lambda \setminus \mu}(p_1, p_2, \dots).
\end{equation}

This expression can be further simplified using the following lemma.

A partition is called a {\em hook} if it has the form $\mu_{a,b} =(a \, 1^b)$ for $a,b \geq 0$. This name is due to the shape of the corresponding Young diagram.

\begin{lemma} \label{Lem4}
We have
\begin{equation}
\frac1{\zeta(z)} s_\mu(\zeta(z), \zeta(2z), \dots) = (-1)^b e^{(a-b-1)z/2}
\end{equation}
if $\mu = \mu_{a,b}$ and 
$s_\mu(\zeta(z), \zeta(2z), \dots) = 0$ otherwise.
\end{lemma}

We will prove this lemma later; at present we continue to simplify the right-hand side of the equality of Lemma~\ref{Lem3}. Using Lemma~\ref{Lem4} we get
\begin{equation}
{\rm RHS} = \sum_{a+b=K}
(-1)^b e^{(a-b-1)z/2} s_{\lambda \setminus \mu}(p_1, p_2, \dots).
\end{equation}

Now let us explain why this coincides with the left-hand part. The vector $\tE_K(z)(\mu)$ has a nonzero $v_\emptyset$ component only when $\mu$ is a hook partition. We have $\tE_K(z) (\mu_{a,b}) = (-1)^b e^{(a-b-1)z/2} v_\emptyset$.
We also need to recall (see, for instance,~\cite{VerOko}) that the Schur polynomial of the representation 
$\lambda \setminus \mu$ of $S_{N-K}$ is obtained by the Murnaghan-Nakayama rule, that is, it is equal to 
\begin{equation}
s_{\lambda \setminus \mu}(p_1,p_2, \dots) = \left< v_\mu \exp \left(\sum_{k \geq 1} \alpha_k p_k/k \right) v_\lambda
\right>.
\end{equation}
Therefore we have
\begin{align}
{\rm LHS} & = \sum_\mu \langle v_\emptyset \tE_K(z) v_\mu \rangle
\left< v_\mu \exp \left(\sum_{k \geq 1} \alpha_k p_k/k \right) v_\lambda
\right>
\\ \notag
& = \sum_{a+b=K}
(-1)^b e^{(a-b-1)z/2} s_{\lambda \setminus \mu_{a,b}}(p_1, p_2, \dots).
\end{align}
\end{proof}

Now we prove Lemma~\ref{Lem4}. 

\begin{proof}
We use the well-known identity for Schur polynomials
\begin{equation}
p_k s_\mu = \sum \pm s_\lambda,
\end{equation}
where the sum is taken over the Young diagrams $\lambda$ obtained by adding a ribbon of length $k$ to $\mu$ and the sign $\pm$ is the parity of the number of downward steps in the ribbon. (This is a reformulation of the action of the operator $\alpha_{-k}$.) First let us check that this formula is compatible with the claim of the theorem.

If $\mu$ is not a hook, then neither of the $\lambda$'s will be a hook. So after the substitution $p_k = \zeta(kz)$ we get the correct equality
\begin{equation}
\zeta(kz) \cdot 0 = \sum \pm 0. 
\end{equation}

If $\mu = \mu_{a,b}$ then there are exactly two ways to add a $k$-ribbon to $\mu$ in such a way that it remains a hook: we can increase either $a$ or $b$ by~$k$. In the first case the sign of the ribbon is $+1$, in the second case it is $(-1)^{k-1}$. Thus we get the correct equality
\begin{equation}
\zeta(kz) \cdot  (-1)^b e^{(a-b-1)z/2} = 
 (-1)^b e^{(a+k-b-1)z/2} + (-1)^{k-1} (-1)^{b+k} e^{(a-b-k-1)z/2}.
\end{equation}

Now, every Schur polynomial can be obtained as a linear combination of the form
\begin{equation}
\sum_i c_i p_{k_i} s_{\lambda_i}.
\end{equation}
For instance, we can use the formula
\begin{equation}
s_\mu = \frac1{|\mu|} \sum p_k \frac{k \; \partial s_\mu}{\partial p_k}
\end{equation}
and decompose every $k \,  \partial s_\mu/\partial p_k$ into a linear combination of Schur polynomials. These expressions allow us to find the Schur polynomials before or after the substitution $p_k = \zeta(kz)$ by induction on the degree. The equality
\begin{equation}
\frac1{\zeta(z)} s_1(\zeta(z), \zeta(2z), \dots) = \frac{\zeta(z)}{\zeta(z)} = 1
= (-1)^0 e^{1-0-1}
\end{equation}
provides the base of induction. Since we know these inductive relations are compatible with the formula given in the lemma, we conclude that the lemma is true.
\end{proof}

Finally, Theorem~\ref{Thm:cutandjoin2} follows immediately from Lemmas~\ref{Lem1}, \ref{Lem2},~\ref{Lem3}. Indeed, in the expression~\ref{Eq:dhdbeta} for $\d H_{r+1}/\d \beta$, the operator $F_{r+1}$ is the coefficient of $z^{r+1}$ in $\tE_0(z)$. Using Lemma~\ref{Lem1} we obtain that $\d H_{r+1}/\d \beta$ is the coefficient of $z^{r+1}$ in
\begin{align}
& \sum_{n \geq 0} \frac1{n!} \sum_{k_1, \dots, k_n} \prod_{i=1}^n 
\frac{\zeta(k_iz) p_{k_i}}{k_i}
\\ \notag &
\times \left< \tE_{\sum k_i}(z) \exp\left(\sum_{k \geq 1} p_k \alpha_k /k\right) \exp(\beta F_{r+1}) \exp\left(\sum_{k \geq 1} q_k \alpha_{-k}/k\right) \right>.
\end{align}
According to Lemma~\ref{Lem3} this is equal to
\begin{align}
&
\sum_{n \geq 0} \frac1{n!} \sum_{k_1, \dots, k_n} \prod_{i=1}^n 
\frac{\zeta(k_iz) p_{k_i}}{k_i} O_{\sum k_i} (z)
\\ \notag & 
\times 
\left<\exp\left(\sum_{k \geq 1} p_k \alpha_k /k\right) \exp(\beta F_{r+1}) \exp\left(\sum_{k \geq 1} q_k \alpha_{-k}/k\right) \right>
\\ \notag &
= \sum Q_{r+1} z^{r+1} H_{r+1},
\end{align}
where in the last equality we have used Lemma~\ref{Lem2} for the expression of $Q_{r+1}$ and Equation~\eqref{Eq:H} for $H_{r+1}$. Extracting the coefficient of $z^{r+1}$, we get 
\begin{equation}
\frac{\d H_{r+1}}{\d \beta} = Q_{r+1} H_{r+1}
\end{equation}
as claimed.

%%%%%%%%%%%%%%%%%%
%%%%%%%%%%%%%%%%%%
%%%%%%%%%%%%%%%%%%

%%%%%%%%%%%%%%%%%%
%%%%%%%%%%%%%%%%%%
%%%%%%%%%%%%%%%%%%

%%%%%%%%%%%%%%%%%%
%%%%%%%%%%%%%%%%%%
%%%%%%%%%%%%%%%%%%

\section{Strong piecewise polynomiality}\label{sectionStrongPiecewisePolynomiality}

In this section, we prove an analogue of strong piecewise polynomiality for Hurwitz numbers with completed $(r+1)$-cycles and derive the wall crossing formulas for this piecewise polynomial. It is a generalization of Johnson's results in~\cite{Joh10}. 

\subsection{Notation}

Throughout this section we fix two positive integers $m$ and $n$. Let $V$ be the subset of $(\Z_{\geq 0})^m \oplus (\Z_{\geq 0})^n$ defined by
\begin{equation}
V:= \left\{(x_1,\dots,x_n,y_1,\dots,y_m) \,\left|\, \sum_{i=1}^n x_i = \sum_{j=1}^m y_j \right.\right\}
\end{equation}
We consider double Hurwitz number with $s$ completed $(r+1)$-cycles as a function $h^s\colon V\to \Q$ such that $h^s(\mu,\nu)=h^s_{\mu,\nu}$

\begin{definition}
Let $I \subset [m]$ and $J \subset [n]$ be any non-empty proper subsets. Then the hyperplane 
\begin{equation}
\{(x,y) \in V\ | \ x_I - y_J = 0 \} \subset V 
\end{equation}
is called the hyperplane given by $I$ and~$J$ and denoted~$W_{I,J}$. 
\end{definition}

\begin{remark}\label{hyperplaneArrangement}
Consider $(\mu, \nu) \in V$ such that it doen't lie on any of the hyperplanes~$W_{I,J}$. Then $h^s_{\mu,\nu}=h^{s,\circ}_{\mu,\nu}$, since there are no covers of $\CP1$ with ramification over $0$ and~$\infty$ given by $\mu$ and~$\nu$ with more than one connected component. Thus, if we interpret the $W_{I,J}$ as the walls of a hyperplane arrangement, then at the internal points of the chambers of this arrangement the disconnected and connected Hurwitz numbers are equal. 
\end{remark}

\subsection{Polynomiality in a chamber}

\begin{theorem}\label{PiecewisePolynomial}
The function~$h^s\colon V \to \Q$ is a piecewise polynomial function with the walls given by the hyperplanes~$W_{I,J}$. 
\end{theorem}

\begin{proof} The proof is analogous to the one in~\cite{Joh10} of piecewise polynomiality for ordinary double Hurwitz numbers. Let $\mathfrak{c}$ be some chamber of the hyperplane arrangement mentioned in Remark~\ref{hyperplaneArrangement}. We have to prove that $h^s|_\mathfrak{c}$ is polynomial. 

The crucial point is that the set of commutation patterns $CP(\mu,\nu)$ does not depend on $(\mu,\nu) \in \mathfrak{c}$, but only on the chamber~$\mathfrak{c}$ itself. Note that for any subsets $I \subset [m]$ and $J \subset [n]$ the sign of the number $\mu_I - \nu_J$ is determined by the chamber~$\mathfrak{c}$ containing~$(\mu, \nu)$, and it is these signs which determine the set of commutation patterns for~$(\mu,\nu)$. Thus, $CP(\mu,\nu)$ depends only on the chamber~$\mathfrak{c}$ containing~$(\mu, \nu)$. From now on, we will denote it by~$CP(\mathfrak{c})$. 

Furthermore, since we are in a chamber of the hyperplane arrangement, by Remark~\ref{hyperplaneArrangement}, the connected and disconnected Hurwitz numbers are equal, so $|T(P)| = 1$ for any commutation pattern in~$CP(\mathfrak{c})$. Thus, the Hurwitz number is determined by Theorem~\ref{connectedTheorem} instead of Theorem~\ref{disconnectedTheorem}. Let us prove that the factor
\begin{equation}
\frac{1}{\zeta(z_{[s]})} \sum_{P \in CP(\mathfrak{c})} \prod_{l\in L(P)}\mZeta
 {P_I^l} {P_J^l}{P_K^l}{P_M^l}{P_N^l}{P_R^l}
\end{equation}
in Equation~\eqref{eq:connected} restricted to~$\mathfrak{c}$ is a power series in $z_1, \ldots z_s$ with coefficients depending polynomially on~$(\mu,\nu)$. Indeed, the only problem is the factor $1/\zeta(z_{[s]})$. Meanwhile, in any commutation pattern in~$CP(\mathfrak{c})$ the last commutator taken will produce a factor $\zeta(nz_{[s]})$ for some integer~$n$. Since $\zeta(nz)/\zeta(z)$ is holomorphic at $z=0$, the whole expression is indeed a power series in $z_1, \dots, z_s$. Clearly, the coefficients depend polynomially on $(\mu,\nu)$. 

Therefore, the coefficient of $z_1^{r+1}\cdots z_s^{r+1}$ will depend polynomially on $\mu$ and~$\nu$ and it only remains to show that it contains a factor $\prod_{i=1}^{l(\mu)} \mu_i \prod_{j=1}^{l(\nu)} \nu_j$. Indeed, in the vacuum expectation value at the start of the algorithm (equation~\ref{algorithmVacuumExpectation}) we have the operator $\mathcal{E}(\{i\},\emptyset,\emptyset)$ for all $i=1,\dots,l(\mu)$. In any commutation pattern $P \in CP(\mathfrak{c})$ it will eventually be commuted with some operator, which will provide a factor $\zeta(\mu_i z_L)$ for a certain subset $L \subset [s]$ that is divisible by $\mu_i$. The same argument also works for $\nu_j$, for all $j=1,\dots,l(\nu)$, and we see from this argument that the product $\prod_{i=1}^{l(\mu)} \mu_i \prod_{j=1}^{l(\nu)} \nu_j$ also devides the product of $\zeta$-functions on the right side of Equation~\eqref{eq:connected}. Thus, $h_g$ is polynomial on any chamber $\mathfrak{c}$, which proves the theorem. 
\end{proof}

\subsection{The structure of the polynomial}

\begin{theorem}
Let $\mathfrak{c}$ be a chamber of the hyperplane arrangement of Remark~\ref{hyperplaneArrangement}. Then $h^s|_{\mathfrak{c}}$ has the following form:
\begin{equation}
h^{s}|_\mathfrak{c}(\mu,\nu) = \sum_{k=0}^g (-1)^k P_{\mathfrak{c},k}^{s}(\mu,\nu) ,
\end{equation}
where $P_{\mathfrak{c},k}^{s}\colon V \to \Q$ is a homogeneous polynomial of degree $(r+1)s + 1 - l(\mu) - l(\nu) - 2k$ with $P_{\mathfrak{c},k}^{s}(\mu,\nu) > 0$ for all $(\mu,\nu) \in \mathfrak{c}$, and $g=(rs - l(\mu) - l(\nu) + 2)/2$ is the genus of the covering. 
\end{theorem}

\begin{proof} Since we know that $h^{s}|_\mathfrak{c}$ is a polynomial, we only need to prove that it is either even or odd, that it is of degree $(r+1)s + 1 - l(\mu) - l(\nu)$, that the lowest occuring term is of degree $2g$ less than the highest, and the alternating nature of the homogeneous terms. It is clear that $h^{s}|_\mathfrak{c}$ is either even or odd from Theorem~\ref{connectedTheorem} and the fact that $\zeta(z)$ is an odd function. 

Let $P \in CP(\mathfrak{c})$ be a commutation pattern. Then for any $l \in L(P)$ we will always have that either $P_K^l$ or $P_R^l$ is non-empty (it is easy to see the vacuum expectation will be zero otherwise, contradicting that $P$ is a commutation pattern). Therefore 
$\mZeta{P_I^l}{P_J^l}{P_K^l}{P_M^l}{P_N^l}{P_R^l}$ has equal total order in $(\mu,\nu)$ and~$(z_1, \dots z_s)$. So the highest total order in $(\mu,\nu)$ in~$h^{s}|_{\mathfrak{c}}$ occurs when we take the lowest possible total order in~$(z_1,\dots,z_s)$ in the factor~$1/\zeta(z_1 + \cdots + z_s)$ in Equation~\eqref{eq:connected} which is equal to $-1$. Thus, the highest occuring order in~$(\mu,\nu)$ is the total degree in $(z_1, \dots, z_s)$ (which is $s(r+1)$), plus~$1$, minus the degree in $(\mu,\nu)$ of $\prod_{i=1}^{l(\mu)} \mu_i \prod_{j=1}^{l(\nu)} \nu_j$, for a total degree of $(r+1)s + 1 - l(\mu) - l(\nu)$.  

The lowest degree term in~$(\mu,\nu)$ of $h^{s}|_\mathfrak{c}$ occurs when we take the lowest possible degree in~$z$ in 
$\mZeta{P_I^l}{P_J^l}{P_K^l}{P_M^l}{P_N^l}{P_R^l}$ for all $l \in L(P)$ which is equal to~$1$. Therefore, the lowest occuring degree in~$(\mu,\nu)$ is equal to the number of commutator terms taken in~$P$. Taking a commutator term reduces the number of $\mathcal{E}$-operators by one. We start with $l(\mu)+l(\nu)+s$ of these operators and we end up with one. Therefore, the number of commutator terms taken is equal to $l(\mu) + l(\nu) + s - 1$. Thus, the lowest degree in~$(\mu,\nu)$ occuring is $l(\mu) + l(\nu) + s - 1 - l(\mu) - l(\nu) = s - 1$. 

By the Riemann-Hurwitz formula, the difference between the highest and the lowest degree in~$(\mu,\nu)$ is then equal to 
$(r+1)s + 1 - l(\mu) - l(\nu) - s + 1 = 2g$.

The coefficients of the homogeneous summands of the expansion of $1/\zeta(z_{[s]})$ have alternating signs. Therefore, to prove positivity of $P_{\mathfrak{c},k}^{s}$ it is enough to prove that all coefficients of odd-degree terms in the expansion of the $\zeta$-functions in the product in Equation~\eqref{eq:connected} are positive. By definition of the algorithm, we can only get a factor 
\begin{equation}
\mZeta{P_I^l}{P_J^l}{P_K^l}{P_M^l}{P_N^l}{P_R^l} = \zeta \left(\det \left( {\begin{smallmatrix}
 |\mu_{P_I^l}| - |\nu_{P_J^l}| & z_{P_K^l} \\
 |\mu_{P_M^l}| - |\nu_{P_N^l}| & z_{P_R^l} 
 \end{smallmatrix} } \right) \right)
\end{equation}
when $\mathcal{E}(P_I^l,P_J^l,P_K^l)$ has positive energy and $\mathcal{E}(P_M^l,P_N^l,P_R^l)$ has negative energy. But then 
\begin{equation}
\det \left( {\begin{array}{ccc}
 |\mu_{P_I^l} - |\nu_{P_J^l}| & z_{P_K^l} \\
 |\mu_{P_M^l} - |\nu_{P_N^l}| & z_{P_R^l} 
 \end{array} } \right) = az_{P_K^l} + b z_{P_R^l} ,
\end{equation}
for some $a,b > 0$, and since the odd-degree coefficients in the expansion of $\zeta(z)$ are all positive, this shows that all coefficients of odd-degree terms in the expansion in~$z_1, \ldots, z_s$ of the $\zeta$-functions in the product in Theorem~\ref{connectedTheorem} are positive. This completes the proof of the theorem.
\end{proof}

\subsection{Wall-crossing formula}

In this section we obtain the wall crossing formula for double Hurwitz numbers with completed cycles with respect to the walls~$W_{I,J}$ described in Remark~\ref{hyperplaneArrangement}. 

Fix $I \subset [m]$ and $J \subset [n]$. Let $\mathfrak{c}_1$ and $\mathfrak{c}_2$ be the two chambers bordering along the wall~$W_{I,J}$. The wall crossing formula is a formula for the difference between the polynomials describing the Hurwitz numbers on the different chambers:
$WC^{(r)}_{I,J} = h^{(r)}_g|_{\mathfrak{c}_1}- h^{(r)}_g|_{\mathfrak{c}_2}$.

In order to compute it, we define a series which captures the information about double Hurwitz numbers with completed $(r+1)$-cycles for any value of~$r$ with given ramification over $0$ and~$\infty$:
\begin{equation}\label{WallCrossingSeries}
H_{\mu,\nu}^s(z_1,\dots,z_s) := \frac{1}{\prod_{i=1}^{l(\mu)} \mu_i \prod_{j=1}^{l(\nu)} \nu_j} \langle \prod_{i=1}^{l(\mu)} \mathcal{E}_{\mu_i}(0) \prod_{k=1}^s \mathcal{E}_0(z_k)\prod_{i=1}^{l(\mu)} \mathcal{E}_{-\nu_j}(0) \rangle.
\end{equation}

\begin{remark} The information of the double Hurwitz number with completed $(r+1)$-cycles and ramification given by $(\mu,\nu)$ is encoded in $H_{\mu,\nu}^s$ for any value of~$r$, that is, $h_{\mu,\nu}^{(r),s}=[z_1^{r+1} \cdots z_s^{r+1}] H_{\mu,\nu}^s$. However, the coefficients of other monomials in $z_1,\dots,z_s$ also have an interpretation as some Hurwitz numbers. The coefficient of the monomial $z_1^{r_1+1} \cdots z_s^{r_s+1}$ for any non-negative integers $r_1, \cdots r_s$ is the number of covers of~$\CP1$ with ramification over $0$ and~$\infty$ given $\mu$ and~$\nu$, and ramification over~$s$ more points given by $\overline{(r_1)}, \ldots, \overline{(r_s)}$.

It is clear from the proof of Theorem~\ref{PiecewisePolynomial} that all coefficients of the power series $H_{\mu,\nu}^s$ are piecewise polynomial with respect to the descriped hyperplane arrangement. Thus, if we take completed cycles of different values for different ramification points, the corresponding Hurwitz number will still be piecewise polynomial.  
\end{remark}

So, let $W_{I,J}$ be a given wall in the hyperplane arrangement of Remark~\ref{hyperplaneArrangement}. Let $\mu, \nu$ also be given. Let $\delta$ denote the difference $\delta := |\mu_I| - |\nu_J|$.

\begin{theorem}
The wall crossing formula is given by
\begin{align}\label{eq:wc}
WC_{I,J}^{(r)}(\mu,\nu) & = [z_1^{r+1} \cdots z_s^{r+1}] \sum_{K \subset [s]} \delta^2 \frac{\zeta(z_K)\zeta(z_{K^c})\zeta(\delta z_{[s]})}{\zeta(\delta z_k) \zeta(\delta z_{K^c})\zeta(z_{[s]})} \cdot\\ \notag
&\cdot H_{\mu_I, \nu_J + \delta}^{|K|}(\{z_k\}_{k \in K}) H_{\mu_{I^c} + \delta, \nu_{J^c}}^{|K^c|}(\{z_k\}_{k \in K^c}). 
\end{align}
\end{theorem}

\begin{proof} 
Let $P$ be a commutation pattern in $CP(\mathfrak{c}_1)$. If $P$ does not produce the operator $\mathcal{E}(I,J,K)$ for some $K \subset [s]$ at some point, it will also be a commutation pattern in $CP(\mathfrak{c}_2)$. Thus, a commutation pattern $P$ only contributes to the wall crossing formula if at some point it produces $\mathcal{E}(I,J,K)$ for some $K \subset [s]$. Let $P$ be such a pattern. Using that the operators in any of the three products in Equation~\eqref{algorithmVacuumExpectation} commute amongst themselves, we may start the algorithm with the vacuum expectation value
\begin{equation}
\left\langle \prod_{i \notin I} \mathcal(i,\emptyset,\emptyset) \prod_{i \in I} \mathcal{E}(i, \emptyset,\emptyset) \prod_{k=1}^s \mathcal{E}(\emptyset, \emptyset, k)\prod_{j \in J}\mathcal{E}(\emptyset,j,\emptyset)\prod_{j \notin J}(\emptyset,j,\emptyset) \right\rangle. 
\end{equation}
If a pattern produces $\mathcal{E}(I,J,K)$, the first vacuum expectation value where it occurs must be
\begin{equation}\label{firstOccur}
\langle \prod_{i \notin I}\mathcal{E}(i,\emptyset,\emptyset) \prod_{k \notin K} \mathcal{E}(\emptyset, \emptyset, k)\ \mathcal{E}(I,J,K) \prod_{j \notin J}\mathcal{E}(\emptyset, j, \emptyset) \rangle. 
\end{equation}
Let $T_1^K$ be the product of $\zeta$-functions produced by the algorithm up untill this point. Note that it does not depend on whether we run the algorithm on $\mathfrak{c}_1$ or~$\mathfrak{c}_2$. It is easy to see that it is given by
\begin{equation}
T_1^K(\{z_k\}_{k \in K}) := \left\langle \prod_{i \in I}\mathcal{E}(i,\emptyset,\emptyset)\prod_{k \in K} \mathcal{E}(\emptyset,\emptyset,k) \prod_{j \in J}\mathcal{E}(\emptyset,j,\emptyset) \mathcal{E}_{-\delta}(0)\right\rangle \frac{\zeta(z_K)}{\zeta(\delta z_K)} . 
\end{equation}
On the other hand, by defintion of the function~$H$ we have:
\begin{align}
& H_{\mu_I,\nu_J + \delta}(\{z_k\}_{k \in K}) = \frac{1}{\delta \prod_{i \in I} \mu_i \prod_{j \in J} \nu_j} \cdot \\ \notag
&\cdot \left\langle \prod_{i \in I}\mathcal{E}(i,\emptyset,\emptyset)\prod_{k \in K} \mathcal{E}(\emptyset,\emptyset,k) \prod_{j \in J}\mathcal{E}(\emptyset,j,\emptyset) \mathcal{E}_{-\delta}(0)\right\rangle.
\end{align}
Therefore,
\begin{equation}\label{eq:t1}
T_1^K(\{z_k\}_{k \in K}) = \delta \prod_{i \in I} \mu_i \prod_{j \in J} \nu_j \frac{\zeta(z_K)}{\zeta(\delta z_K)} H_{\mu_I,\nu_J + \delta}(\{z_k\}_{k \in K}).
\end{equation}

Let $T_2^K$ denote the difference between the polynomials computing the vacuum expectation value~\eqref{firstOccur} on $\mathfrak{c}_1$ and~$\mathfrak{c_2}$. To compute this, it will be better to let the algorithm run according to the rules of the chamber~$\mathfrak{c}_1$ on both sides (i.e.; we use the set of commumation patterns $CP(\mathfrak{c}_1)$). This means we will move the operator $\mathcal{E}(I, J, K)$ to the left on both chambers, even though on~$\mathfrak{c}_2$ it has positive energy.

If the operator $\mathcal{E}(I, J, K)$ is involved in a commutator term at any point, the algorithm will run as normal afterwards on both chambers, since the chambers differed by only one wall. Therefore, the only contribution to $T_2^K$ comes from commutation patterns where $\mathcal{E}(I, J, K)$ is moved entirely from to left. The result will then be zero on~$\mathfrak{c}_1$, since there is an operator of negative energy on the far left, but it will be non-zero on~$\mathfrak{c}_2$. The last step in the algorithm on $\mathfrak{c}_2$ for Equation~\eqref{firstOccur} will be
\begin{equation}
\langle \mathcal{E}(I, J, K)\ \mathcal{E}(I^\text{c}, J^\text{c}, K^\text{c})\rangle = \frac{\zeta(\delta z_{[s]})}{\zeta(z_{[s]})} ,
\end{equation}
where $I^\text{c}$ denotes the complement of $I \subset [m]$, and the same for $J^\text{c}$ and~$K^\text{c}$. Also using that
\begin{equation}
\langle \mathcal{E}_\delta(0)\ \mathcal{E}(I^\text{c}, J^\text{c}, K^\text{c}) \rangle = \frac{\zeta(\delta z_{K^\text{c}})}{\zeta(z_{K^\text{c}})},
\end{equation}
we see that 
\begin{align}
& T_2^K(\{z_k\}_{k \notin K}) \\ \notag 
&= \langle\mathcal{E}_\delta(0) \prod_{i \notin I}\mathcal{E}(i,\emptyset,\emptyset)\prod_{k \notin K}\mathcal{E}(\emptyset, \emptyset, k)\prod_{j \notin J} \mathcal{E}(\emptyset, j, \emptyset)\rangle \frac{\zeta(z_{K^\text{c}}) \zeta(\delta z_{[s]})}{\zeta(\delta z_{K^\text{c}})\zeta(z_{[s]})} .
\end{align}
(by $\prod_{k \notin K}$ we denote $\prod_{k \in K^\text{c}}$; the same for $I$ and $J$).
On the other hand
\begin{align}
& H_{\mu_{I^\text{c} + \delta},\nu_{J^\text{c}}}(\{z_k\}_{k \in K^\text{c}}) = \frac{1}{\delta \prod_{i \notin I} \mu_i \prod_{j \notin J} \nu_j} \cdot \\ \notag
&\cdot \langle \mathcal{E}_\delta(0)\ \prod_{i \notin I}\mathcal{E}(i,\emptyset,\emptyset)\prod_{k \notin K} \mathcal{E}(\emptyset,\emptyset,k) \prod_{j \notin J}\mathcal{E}(\emptyset,j,\emptyset) \mathcal{E}_{-\delta}(0)\rangle,
\end{align}
therefore
\begin{align}\label{eq:t2}
& T_2^K(\{z_k\}_{k \notin K}) \\ \notag
& = \delta \prod_{i \notin I} \mu_i \prod_{j \notin J} \nu_j \mathfrak{\zeta(z_{K^\text{c}})\zeta(\delta z_{[s]})}{\zeta(\delta z_{K^\text{c}})\zeta(z_{[s]})} H_{\mu_{I^\text{c}} + \delta, \nu_{J^\text{c}}}(\{z_k\}_{k \in K^\text{c}}) .
\end{align}
It is clear that 
\begin{equation}
WC_{I,J}^{(r)}(\mu, \nu) = \frac{1}{\prod_{i \in [m]} \mu_i \prod_{j \in [n]} \nu_j} [z_1^{r+1} \cdots z_s^{r+1}] \sum_{K \subset [s]} T_1^K T_2^K .
\end{equation}
Substituting Equations~\eqref{eq:t1} and~\eqref{eq:t2} into this formula, we obtain the wall crosing formula~\eqref{eq:wc}.
\end{proof}

%%%%%%%%%%%%%%%%%%
%%%%%%%%%%%%%%%%%%
%%%%%%%%%%%%%%%%%%

%%%%%%%%%%%%%%%%%%
%%%%%%%%%%%%%%%%%%
%%%%%%%%%%%%%%%%%%

%%%%%%%%%%%%%%%%%%
%%%%%%%%%%%%%%%%%%
%%%%%%%%%%%%%%%%%%

\section{An analogue of GJV-formula}

In this section, we discuss an analogue of the Goulden-Jackson-Vakil formula for the one-part double Hurwitz numbers that might relate them to the intersection theory of some moduli spaces. These conjectural ``intersection numbers'' have very nice properties as some explicit solutions of the KP hierarchy. The number $r\geq 1$ is fixed throughout the section, so we omit the superscript $(r)$ in all notations.

\subsection{The formula} Let $\mu$ be an arbitrary partition $(\mu_1,\dots,\mu_l)$ and $g\geq 0$ be an arbitrary non-negative integer. We consider the one-part double Hurwitz number with completed $(r+1)$-cycles $h_{g,|\mu|,\mu}$.
We propose the following formula:
\begin{equation}\label{eq:GJV}
h_{g,|\mu|,\mu} = \frac{m!}{d} \int_{X_{g,n}} \frac{1-\Lambda_{2}+\Lambda_{4}-\cdots+(-1)^{g}\Lambda_{2g}}
{(1-\mu_1\Psi_1)\cdots(1-\mu_n\Psi_n)},
\end{equation}
where $X_{g,n}$ is a sequence of spaces of complex dimension $2g(r+1)+n-1$, and we fix the degrees of the rational cohomology classes 
$\Lambda_{2k}\in H^{4rk}(X_{g,n})$ and $\Psi_1,\dots,\Psi_r\in H^{2r}(X_{g,n})$. Existence of these geometric objects is a pure speculation, so a way to understand this formula is the following. 

One-part double Hurwitz numbers with completed cycles $h_{g,|\mu|,\mu}$ are expressed in terms of some new combinatorially significant numbers that we denote by 
\begin{equation}
\langle\Lambda_{2k}\prod_{i=1}^n \tau_{d_i}\rangle_g := \int_{X_{g,n}}\Lambda_{2rk}\prod_{i=1}^n\Psi_i^{d_i}
\end{equation}
that are symmetric in $d_1,\dots,d_n$, non-zero only if $2g(r+1)+n-1=(2k+\sum_{i=1}^nd_i)r$, and have interesting properties together with a hope to be related to geometry in future.

\subsection{Generating function for intersection numbers} We consider a generating function for the numbers $\langle\Lambda_{2k}\prod_{i=1}^n \tau_{d_i}\rangle_g$.
Let
\begin{equation}
G(u):= \sum_{j,k_1,k_2,\dots} (-1)^{j}\langle \Lambda_{2j} \tau_0^{k_0} \tau_1^{k_1} \dots\rangle_g u^{2j}
\frac{T_0^{k_0}}{k_0!} \frac{T_1^{k_1}}{k_1!} \dots.
\end{equation}
Here we take the sum of all non-negative integer indices $j,k_1,\dots,k_n$, $n\geq 0$, such that there exists a non-negative integer $g$ such that 
$2g(r+1)+n-1=(2k+\sum_{i=1}^nd_i)r$.

\begin{notation}\label{notationAbuse}
In this section, we use the isomorphism described in section~\ref{sectionCutAndJoin} between the infinite wedge space and the space of formal power series $\C[[q_1, q_2, \ldots]]$ to interpret the $\mathcal{E}$-operators as operators on $\C[[q_1, q_2, \ldots]]$. By abuse of notation, we denote these operators in the same way. We denote by $\mathcal{E}_{k,a}$ the operator $[z^a] \mathcal{E}_k(z)$. 
\end{notation}

We would like to consider the formal variable $T_k$, $k=0,1,\dots$, as linear functions in formal variables $q_i$, $i=1,2,\dots$. We set $T_0=q_1$, and $T_{k+1}=(u\E_{0,1}+\E_{-1,1}) T_k$. We list the first few expressions:
\begin{align}\label{eq:T}
T_0 &= q_1,\\
T_1 &= uq_1 + q_2, \notag \\
T_2 &= u^2 q_1 + 3uq_2 + 2q_3, \notag
\end{align}
and so on. 

\begin{theorem}\label{thm1} For any function $c(u)$, the series 
$c(u)+G(u,q_1,q_2,\dots)$
is a solution of the Hirota equations in variables $q_i$, $i=1,2,\dots$ ($u$ is just a parameter).
\end{theorem}

In particular, we consider the series $F:=G|_{u=0}$, that is, the generation function for the intersection numbers without $\Lambda$-classes:
\begin{equation}
F(q_0,q_1,\dots)=\sum_{k_0,k_1,\dots} \langle \tau_0^{k_0} \tau_1^{k_1} \dots\rangle_g 
\frac{(0!q_1)^{k_0}}{k_0!} \frac{(1!q_2)^{k_1}}{k_1!} \dots.
\end{equation}
Theorem~\ref{thm1} implies the following property of $F$.

\begin{corollary} The series $F(q_1,q_2,\dots)$ is a solution of the Hirota equations and linearized Hirota equations.
\end{corollary}

\subsection{An explicit formula for $G$ and $F$} In this section, we give explicit formulas for the series $G$ and $F$. For that we need to introduce some operators $Y_i$, $i\geq 0$. We denote by $Y_0$ the operator $\tilde\E_0$. We denote by $Y_{i+1}$, $i\geq 0$, the operator
\begin{equation}
Y_{i+1}(w):=\zeta(w)^{i+1}\left(\prod_{k=0}^i \left(\frac{\d}{\d w} - \frac{i}{2}+k\right)\right)\E_{-(i+1)}(w)
\end{equation}
Observe that $Y_{i+1}(w)=O(w^{i+1})$.

\begin{theorem}\label{thm2} We have:
\begin{equation}\label{eq:tau2}
G(u,q_1,q_2,\dots)=\exp\left([w^{r+1}]\sum_{k=0}^{r+1} u^k \frac{Y_{r+1-k}(w)}{(r+1-k)!}\right)q_1.
\end{equation}
\end{theorem}

\begin{corollary} We have:
\begin{equation}\label{eq:formula-F}
F(q_1,q_2,\dots)=\exp\left(Y_{r+1}(0)\right)q_1.
\end{equation}
\end{corollary}

\begin{remark}
The constant term of $Y_{r+1}$ (used in the Equation~\eqref{eq:formula-F}) is a linear combination of $\E_{-(r+1),i}$, where $i=1,3,5,\dots,r+1$ ($i=0,2,4,\dots,r+1$) for even (respectively, odd) $r$, and the coefficients are \emph{central factorial numbers}~\cite{OEIS}.
\end{remark}

\subsection{Rearranging the generating series}

Consider the following version of a generating series for Hurwitz numbers,
\begin{equation}\label{eq:H}
H(\beta,p_1,p_2,\dots):=\sum_{g,n}  \frac{1}{n!}\sum_{\mu_1,\dots,\mu_n}|\mu|\cdot h_{g,|\mu|,\mu} p_{b_1}\cdots p_{b_n}\frac{\beta^m}{m!}.
\end{equation}
Here $m=(2g+n-1)/r$. In fact, it is the way we rather present a generating series for the integrals in the formula~\eqref{eq:GJV}. Since
\begin{equation}\label{eq:H-formula}
H+c(\beta)=\exp\left(\beta \cdot [w^{r+1}]\tilde\E_0(w)\right)\left(c(\beta)+ \sum_{i=1}^\infty p_i\right),
\end{equation} 
we conclude that $H+c(\beta)$ satisfies the Hirota equations, for an arbitrary function $c(\beta)$. Note that Equation~\eqref{eq:H-formula} is just the cut-and-join equation and that the operator $[w^{r+1}]\tilde\E_0(w)$ is actually explicitly given by the operator $Q_{r+1}$ discussed Section~\ref{sectionCutAndJoin} (see Notation~\ref{notationAbuse}).

Now, using that $(r+1)m=\dim X_{g,n}/r+n-1$, we obtain:
\begin{align}\label{eq:H-expansion}
H&=\sum_{n=1}^\infty \frac{1}{n!} \sum_{g,\mu_1,\dots,mu_n}\int_{X_{g,n}}\frac{1-\Lambda_2+\dots\pm\Lambda_{2g}}{(1-\mu_1\Psi_1)\cdots (1-\mu_n\Psi_n)}p_{\mu_1}\cdots p_{\mu_n}\beta^{m} \\
&= \frac{1}{u}\sum_{n=1}^\infty \frac{1}{n!} \sum_{g,\mu_1,\dots,b\mu_n}\int_{X_{g,n}}(1-u^2\Lambda_2+\dots\pm u^{2g}\Lambda_{2g})\prod_{i=1}^n
\frac{up_{\mu_i}}{(1-u\mu_i\Psi_i)} \notag \\
& = \frac{1}{u}\sum_{g,n}\frac{1}{n!}
\left\langle\left(1-u^2\Lambda_2+u^4\Lambda_4-\dots\right)\prod_{i=1}^n \left(\sum_{d\geq 0} \tau_d T_d\right)\right\rangle_g, \notag
\end{align}
where $u^{r+1}=\beta$ and
$T_d=\sum_{b\geq 1} up_b\cdot (ub)^d=u^{d+1}\sum_{b\geq 1} b^d p_b$.
Observe that $T_{d+1}=u\tilde\E_{0,1}T_d$.

\subsection{Change of variables}

We use the same change of variables as in~\cite{ShaZvo07, Sha09}. We rescale the variables by setting $p_b=q_b/u^b$, and then we replace $q_i$ with $\exp(-\E_{-1,1}/u)q_i$, $i=1,2,\dots$. An explicit formula for this linear triangular change of variables is given by 
\begin{equation}\label{eq:change}
p_b=\sum_{i=b}^\infty \frac{1}{u^i}(-1)^{i-b}\binom{i-1}{b-1}q_i.
\end{equation}
Under this change of variable a series $f(u,p_1,p_2,\dots)$ transforms into $g(u,q_1,q_2,\dots):=\exp(-\E_{-1,1}/u)f(u,q_1/u,q_2/u^2,\dots)$. This change of variable is a symmetry of the Hirota equations.

A straightforward computation shows that $T_0$ turns into $q_1$ and $u\tilde\E_{0,1}$ turns into $\exp(-\E_{-1,1}/u)u\tilde\E_{0,1}\exp(\E_{-1,1}/u)=u\tilde\E_{0,1}+\E_{-1,1}$. This means that $T_d=(u\tilde\E_{0,1}+\E_{-1,1})^dq_1$.

In order to apply the change of variables to the operator $\tilde\E_0(w)$, we need the following lemma.
\begin{lemma} We have: $[Y_i,\E_{-1,1}]=Y_{i+1}$, $i=0,1,2,\dots$.
\end{lemma}
\begin{proof}
Indeed, 
\begin{align}\label{eq:comm-lemma}
[Y_i,\E_{-1,1}] & = \left[\zeta(w)^{i}\left(\prod_{k=0}^{i-1} \left(\frac{\d}{\d w} - \frac{i-1}{2}+k\right)\right)\E_{-(i)}(w),[z]\E_{-1}(z)\right].
\end{align}
Observe that $[z]\left[\E_{-(i)}(w),[z]\E_{-1}(z)\right]=[z]\left(\zeta(w-iz)\E_{-(i+1)}(z+w)\right)=\left(\zeta(w)\frac{\d}{\d w}-i\zeta'(w)\right)\E_{-(i+1)}(w)$. A straighforward computation implies that
\begin{align}
& \left(\frac{\d}{\d w} + \frac{i-1-2k}{2}\right)\left(\zeta(w)\frac{\d}{\d w}-(i-k)\zeta'(w)-\frac{k}{2}\zeta(w)\right) \\
& = \left(\zeta(w)\frac{\d}{\d w}-(i-1-k)\zeta'(w)-\frac{k+1}{2}\zeta(w)\right)\left(\frac{\d}{\d w} + \frac{i-2k}{2}\right),\notag
\end{align}
$k=0,1,\dots,i-1$. Therefore,
\begin{align}
& \left(\prod_{k=0}^{i-1} \left(\frac{\d}{\d w} - \frac{i-1}{2}+k\right)\right)\left(\zeta(w)\frac{\d}{\d w}-i\zeta'(w)\right) \\
& = \zeta(w)\prod_{k=0}^i \left(\frac{\d}{\d w} - \frac{i}{2}+k\right). \notag
\end{align}
Thus we see that the right hand side of Equation~\eqref{eq:comm-lemma} is equal to
\begin{equation}
\zeta(w)^{i}\cdot \left(\zeta(w)\prod_{k=0}^i \left(\frac{\d}{\d w} - \frac{i}{2}+k\right)\right) \E_{-(i+1)}(w) = Y_{i+1}.
\end{equation}
\end{proof}

\begin{corollary}\label{cor:E0} Under the change of variable the operator $\tilde\E_0(w)$ turns into $\sum_{i=0}^\infty u^{-i}Y_i/i!$
\end{corollary}

\subsection{Proof of Theorems~\ref{thm1} and~\ref{thm2}}

The generating series given in Equation~\eqref{eq:H}, $H+c(u)$, is a solution to Hirota equations.
We apply the change of variables~\eqref{eq:change}. Using Corollary~\ref{cor:E0}, we see that this change of variables Equation~\eqref{eq:H-formula} turns into
\begin{equation}
c(u)+\exp\left(u^{r+1}[w^{r+1}]\sum_{i=0}^\infty u^{-i}Y_i/i!\right)\frac{q_1}{u}.
\end{equation}
Using that $Y_i(w)=O(w^{i})$, we see that this formula multiplied by $u$ is equal to the right hand side of the Equation~\eqref{eq:tau2} (if we choose $c(u)=0$). On the other hand, from Equation~\eqref{eq:H-expansion} we know that $H$ multiplied by $u$ is equal to $G$ in coordinates $u,q_1,q_2,\dots$. This completes the proof of Theorem~\ref{thm2}. Theorem~\ref{thm1} is then obvious since the change of variables~\eqref{eq:change} is a symmetry of the KP-hierarchy.

%%%%%%%%%%%%%%%%%%
%%%%%%%%%%%%%%%%%%
%%%%%%%%%%%%%%%%%%

%%%%%%%%%%%%%%%%%%
%%%%%%%%%%%%%%%%%%
%%%%%%%%%%%%%%%%%%

%%%%%%%%%%%%%%%%%%
%%%%%%%%%%%%%%%%%%
%%%%%%%%%%%%%%%%%%


\begin{thebibliography}{00}

\bibitem{OEIS} The On-Line Encyclopedia of Integer Sequences, http://oeis.org/A008955 and http://oeis.org/A008956.

\bibitem{CavJohMar10a} R.~Cavalieri, P.~Johnson, H.~Markwig, Tropical Hurwitz Numbers, J.~Algebr.~Comb.~\textbf{32} (2010), no.~2, 241-–265.

\bibitem{CavJohMar10b} R.~Cavalieri, P.~Johnson, H.~Markwig, Chamber Structure of Double Hurwitz numbers, arXiv:1003.1805v1 

\bibitem{Chi08}
A.~Chiodo, Towards an enumerative geometry of the moduli space of twisted curves and $r$-th roots,
Compos.~Math.~\textbf{144} (2008), no.~6, 1461--1496. 

\bibitem{FanPan02}
B. Fantechi, R. Pandharipande, Stable maps and branch divisors,
Compositio Mathematica, \textbf{130} (2002), p. 345--364.

\bibitem{GouJac} 
I.~P.~Goulden, D.~M.~Jackson, Transitive factorisations into transpositions and holomorphic mappings on the sphere, Proc.~Amer.~Math.~Soc.~\textbf{125} (1997), no.~1, 51–-60. 

\bibitem{GouJacVak05} I.~P.~Goulden, D.~M.~Jackson, R.~Vakil, Towards the geometry of double Hurwitz numbers, Adv. Math.  \textbf{198}  (2005),  no.~1, 43--92.

\bibitem{Jam78} G.~D.~James, The representation theory of the symmetric groups,
Lecture Notes in Math., 682. Springer-Verlag, Berlin-Heidelberg-New York, 1978. 


\bibitem{Joh10} P.~Johnson, Double Hurwitz numbers via the infinite wedge, arXiv:1008.3266.

\bibitem{KacRai87} V.~Kac, A.~K.~Raina, Bombay lectures on highest weight representations of infinite-dimensional Lie algebras, Adv. Ser. Math. Phys., 2. World Scientific, Teaneck, NJ, 1987.

\bibitem{KerOls94} S.~Kerov, G.~Olshanski, Polynomial functions on the set of Young diagrams, C.~R.~Acad.~Sci.~Paris S\'er.~I Math.~\textbf{319} (1994), no.~2, 121-–126.

\bibitem{MiwJimDat00} T.~Miwa, M.~Jimbo, E.~Date, Solitons. Differential equations, symmetries, and infinite-dimensinal algebras, Cambridge Tracts in Math., 135. Cambridge University Press, Cambridge, 2000.

\bibitem{Oko00} A.~Okounkov, Toda equations for Hurwitz numbers, Math.~Res.~Lett.~\textbf{7} (2000), 447--453.

\bibitem{OkoPan06}
A.~Okounkov, R.~Pandharipande, Gromov-Witten theory, Hurwitz theory, and completed cycles,
Ann.~Math.~(2)~\textbf{163} (2006), no.~2, 517--560.

\bibitem{Sha09} S.~Shadrin, 
On the structure of Goulden-Jackson-Vakil formula, 
Math.~Res. Lett.~\textbf{16} (2009), no.~4, 703--710.

\bibitem{Sha05} S.~Shadrin,
Some relations for double Hurwitz numbers,
Funct.~Anal.~Appl.~\textbf{39} (2005), no.~2, 160--162.

\bibitem{ShaShaVai08} S.~Shadrin, M.~Shapiro, A.~Vainshtein,
Chamber behavior of double Hurwitz numbers in genus $0$, 
Adv.~Math.~\textbf{217} (2008), no.~1, 79--96.

\bibitem{ShaZvo07} S.~Shadrin, D.~Zvonkine, Changes of variables in ELSV-type formulas,  Michigan Math.~J.~\textbf{55}  (2007),  no.~1, 209--228.

\bibitem{VerOko} A.~Vershik, A.~Okounkov, A new approach to the representation theory of the symmetric groups, Selecta Math. (N.S.)~\textbf{2} (1996), no.~4, 581–-605.

\bibitem{Zvo06} 
D.~Zvonkine, A preliminary text on the $r$-ELSV formula, preprint 2006.


\end{thebibliography}
\end{document}